\documentclass[12pt, reqno]{amsart}
\usepackage{amsmath,amssymb,amsthm}

\usepackage{amsmath, amssymb}
\usepackage{amsthm, amsfonts, mathrsfs}
\usepackage{mathptmx}
\usepackage{fullpage}
\usepackage{amsfonts,graphicx}
\numberwithin{equation}{section}
\usepackage[colorlinks=true, pdfstartview=FitV, linkcolor=blue, citecolor=blue, urlcolor=blue]{hyperref}

\newcommand{\RR}{\mathbf{R}}

\newcommand{\EE}{\varepsilon}

\newcommand{\Ss}{\mathbb{S}}

\newcommand{\Div}{\textnormal{div }}

\newcommand{\supp}{\textnormal{supp }}

\newcommand{\axi}{\textnormal{axi}}

\newcommand{\essup}{\textnormal{essup}}

\newcommand{\lin}{\textnormal{lin}} 
\newcommand{\BMO}{\textnormal{BMO}} 
\newcommand{\bydef}{\overset{def}{=}}
\newcommand{\norm}[1]{\left\Vert #1\right\Vert}
\newtheorem{Theo}{Theorem}[section]
\newtheorem{lem}[Theo]{Lemma}

\newtheorem{prop}[Theo]{Proposition}

\theoremstyle{plain}
\theoremstyle{definition}
\newtheorem{defi}[Theo]{Definition}

\theoremstyle{remark}
\newtheorem{Rema}[Theo]{Remark}
\newtheorem*{rema*}{Remark}

\author[A. Hanachi]{Adalet Hanachi}
\address{LEDPA, Universit\'e de Batna --2--\\ Facult\'e des Math\'ematiques et Informatique\\ D\'epartement de Math\'ematiques\\ 05000 Batna Alg\'erie}
\email{a.hanachi@univ-batna2.dz}

\author[H. Houamed]{Haroune Houamed}
\address{Department of Mathematics, New York University in Abu Dhabi, Saadiyat Island \\
P.O. Box 129188, Abu Dhabi, United Arab Emirates }
\email{haroune.houamed@nyu.edu}

\author[M. Zerguine]{Mohamed Zerguine}
\address{LEDPA, Universit\'e de Batna --2--\\ Facult\'e des Math\'ematiques et Informatique\\ D\'epartement de Math\'ematiques\\ 05000 Batna Alg\'erie}
\email{m.zerguine@univ-batna2.dz}

\date{}

\thanks{$\bullet$The authors would like to thank Taoufik Hmidi for the fruitful discussion on the subject. }
\thanks{$\bullet$The second author would like to thank Gaëtan Cane for the helpful discussion on Measure theory.}

\begin{document}

\title[Inviscid limit]
{Remarks on the global well-posedness of the  axisymmetric Boussinesq system with rough initial data}

\begin{abstract}This work concerns the global well-posedness problem for the 3D axisymmetric viscous Boussinesq system with critical rough initial data. More precisely, we aim to extending our recent result \cite{Hanachi-Houamed-Zerguine} to the case of initial data of measure type.  To this end, we should first develop some notions of axisymmetric measures in a general context, then, in the spirit of \cite{Gallay-Sverak}, we prove the global wellposedness result provided that the atomic parts of the initial measures are small enough.
\end{abstract}
\maketitle
\tableofcontents

\section{Introduction} 
\subsection{Model and epitome of results} In geophysical fluid dynamics, density variations may arise at low speeds due to the changes in temperature or humidity like in atmosphere, or salinity as in oceans which give rise to buoyancy forces. The effect of these density changes can be expressive even if the fractional change in density is small. The Boussinesq approximation retains density variations in gravity term  responsible for the buoyancy effect but disregards them in the inertial term. This provides a set of equations which are well-known as the Boussinesq system  
\begin{equation}\label{B(mu,kappa)}
\left\{ \begin{array}{ll} 
\partial_{t}v+v\cdot\nabla v-\mu\Delta v+\nabla p=\rho \vec e_z & \textrm{$x\in\RR^3,\quad t\in(0,\infty)$,}\\
\partial_{t}\rho+v\cdot\nabla \rho-\kappa\Delta \rho=0 & \textrm{$x\in\RR^3,\quad t\in(0,\infty)$,}\\ 
\Div v=0, &\\ 
({v},{\rho})_{| t=0}=({v}_0,{\rho}_0). \tag{B$_{\mu,\kappa}$}
\end{array} \right.
\end{equation} 
Usually, $v(t,x)$ refers to the distribution of the fluid velocity localized in $x\in\RR^3$ at a time $t\in(0,\infty)$ with free-divergence, the scalar function $\rho(t,x)$ designates the mass density in the modeling of geophysical fluids and $p(t,x)$ is the force of internal pressure. The non--negative parameters $\mu$ and $\kappa$ represent respectively the kinematic viscosity and molecular diffusivity of the fluid which can be seen as the inverse of Reynolds numbers. On the other hand, the term $\rho\vec{e}_z$ models the influence of the buoyancy force in the fluid motion with respect to the vertical direction $\vec{e}_z=(0,0,1)$.

\hspace{0.5cm}The system \eqref{B(mu,kappa)} strongly ubiquitous in the mathematics community either theoretically or experimentally because it arises in many phenomena like thermal convection, dynamic of geophysical fluids, and optimal mass transport topic, e.g. \cite{Brenier,Ped}.  


\hspace{0.5cm}For better understanding, we embark with Navier-Stokes equations which is coming from \eqref{B(mu,kappa)} when the initial density is constant. The outcome system reads as follows
\begin{equation}\label{NS(mu)}
\left\{ \begin{array}{ll} 
\partial_{t}v+v\cdot\nabla v-\mu\Delta v+\nabla p=0 & \textrm{$x\in\RR^3,\quad t\in(0,\infty)$,}\\
\Div v=0, &\\ 
v_{| t=0}=v_0. \tag{NS$_\mu$}
\end{array} \right.
\end{equation} 
The mathematical issue of the global solutions with finite energy is due to J. Leray in \cite{Leray} who proved that if $v_0\in L^2(\RR^3)$, then, the system \eqref{NS(mu)} admits at one global solution 
\begin{equation*}
v\in L^\infty\big((0,\infty);L^2(\RR^3)\big)\cap L^2\big((0,\infty);\dot H^1(\RR^3)\big).
\end{equation*}
The uniqueness of such solutions remains open, unless in the two-dimensional case. The second main feature of the \eqref{NS(mu)} is the fact that it is invariant under the transformation 
\begin{equation}\label{Sym}
v(t,x)\mapsto v_\lambda(t,x)\triangleq \lambda v(\lambda^2 t,\lambda x).
\end{equation}
 In other words, if $v$ is a solution of \eqref{NS(mu)} on $[0,T]$ with initial data $v_0$, then $v_\lambda$ is a solution on $[0,\lambda^{-2}T]$ with intial data $v_{0_{\lambda}}\triangleq \lambda v_0(\lambda \cdot)$. The first result positive result of \eqref{NS(mu)} in scaling-invariant spaces goes back to Kato \cite{Kato} who proved the local well-posedness of this system for any initial data $v_0\in L^N(\RR^N)$. Yet, the extension of such solutions, globally in time, was asserted only for sufficiently small data in such space. This result was widened lately by a numerous authors in several larger scaling spaces with respect to the following chain embedding 
\begin{equation}\label{Sym2}
L^N(\RR^N)\hookrightarrow \dot B_{p,r}^{-1+\frac{N}{p}}\hookrightarrow \dot B_{\widetilde p,\widetilde r}^{-1+\frac{N}{\widetilde p}}\hookrightarrow\BMO^{-1}   
\end{equation}
for $2\le p\le \widetilde{p}<\infty$ and $2\le r\le \widetilde r\le\infty$. 

\hspace{0.5cm}Let us mention that the topic of blow-up in finite time of smooth solutions with large initial data of \eqref{NS(mu)} is still now not knwon, expect in some partial situation. Actually, Chemin and {\it al.} investigated in series of references \cite{Chemin-Gallaghere0,Chemin-Gallaghere1} that \eqref{NS(mu)} is in fact global in time where initial data which are not small in any critical space but satisfies some structure like oscillations or slow variations in one direction. For a huge litterature about this subject, we refer to \cite{Canone-Planchon,Koch-Tataru,Meyer,Weissler} and references therein.  

\hspace{0.5cm}The question that arises more legitimately is: can we extend the local solution we mentioned above to be a global one? In two dimensional case, the response is positive whenever the initial data belongs to any class from \eqref{Sym2}. For higher dimension, the positive answer remains true as long as the initial data satisfies a suitable smallness condition. Thus, let us shed a light on the 2D case  where the key factor is that the vorticity boils down to scalar function $\omega=\partial_1 v^2-\partial_2v^1$. Through \eqref{NS(mu)}, $\omega$ satisfies the following evolution equation 
\begin{equation}\label{Vort-2d}
\left\{ \begin{array}{ll} 
\partial_{t}\omega+v\cdot\nabla \omega-\mu\Delta \omega =0 & \textrm{$x\in\RR^2,\quad t\in(0,\infty)$,}\\
\omega_{| t=0}={\omega}_0.
\end{array} \right.
\end{equation} 
such that, the corresponding velocity $v$ is recuperated by the so-called {Biot-Savart law}, that is
\begin{equation*}
v(t,x)=\frac{1}{2\pi}\int_{\RR^2}\frac{(x-y)^\perp}{|x-y|^2}\omega(t,x)dx,
\end{equation*}
where $x^\perp=(-x_2,x_1)$. Through scaling invariance \eqref{Sym}, the convenient transformation of the vorticity  is given by
\begin{equation*}
\omega(t,x)\mapsto \lambda^2\omega(\lambda^2 t,\lambda x).
\end{equation*}
An elementary calculus shows that $L^1(\RR^2)$ and $\mathscr{M}(\RR^2)$ (where $\mathscr{M}(\RR^2)$ is the space of Radon measures with finite mass) are  critical spaces for \eqref{Vort-2d}. 

\hspace{0.5cm}As a result, Cottet in \cite{Cottet}, and independently Giga, Miyakawa, and Osada in \cite{giga-miyakawa-osada} have granted a global result when the initial vorticity $\omega_0$ belongs to $\mathscr{M}(\RR^2)$. To reach, however, the uniqueness which is very hard for an arbitrary initial data in $\mathscr{M}(\RR^2)$ they explored a Gronwall-type argument, showing in \cite{giga-miyakawa-osada} that uniqueness involves that the atomic part of $\omega_0$ is small enough. The interpretation of \cite{giga-miyakawa-osada} that the size requirement only entails the atomic part of the measure coming from the pivotal estimate
\begin{equation*} 
\limsup_{t\rightarrow0}t^{1-\frac1p}\|e^{t\Delta}\mu\|_{L^p}\le C_p\|\mu_{pp}\|_{\mathscr{M}(\RR^2)},\quad p\in(1,\infty],
\end{equation*}
with $\|\mu_{pp}\|_{\mathscr{M}(\RR^2)}$ refers to the total variation of the atomic part of $\mu\in\mathscr{M}(\RR^2)$. This latter result was early enhanced by Gallagher and Gallay in \cite{Gallagher-Gallay}, where they established that if \mbox{$\omega_0\in\mathscr{M}(\RR^2)$}, there exists a unique solution $\omega\in C\big((0,\infty);L^1\cap L^\infty\big)$ so have $\|\omega(t,\cdot)\|_{L^1}\le\|\omega_0\|_{\mathscr{M}(\RR^2)}$ and demonstrated also that such solution is in fact continuously dependent on initial data, deducing that Navier-Stokes equations is globally well-posed in 2D case. For large literature we refer the reader to \cite{PG}.

\hspace{0.5cm}For 3D Navier-Stokes equations the classical paradigm \`a la Leray and \`a la Kato remains valid. In terms of vorticity $\omega=\nabla\times v$, the situation is very worse because of the additional term $\omega\cdot\nabla v$ in $\omega$'s equation. In other words, we have 
\begin{equation}\label{Vort-3d}
\left\{ \begin{array}{ll} 
\partial_{t}\omega+v\cdot\nabla \omega-\mu\Delta \omega =\omega\cdot\nabla v & \textrm{$x\in\RR^3,\quad t\in(0,\infty)$,}\\
\omega_{| t=0}={\omega}_0.
\end{array} \right.
\end{equation} 
The appearance of the term $\omega\cdot\nabla v$ is due to the higher dimension and is often referred to as the vorticity {\it stretching term}. Note that for 2D case, we have $\omega\cdot\nabla v\equiv0$, we immediately deduce for $t\ge0$ that $\|\omega(t)\|_{L^p}\le\|\omega_0\|_{L^p}$ for $p\in[1,\infty]$. According to Beale, Kato and Majda criterion \cite{BKM} this latter boundedness is the main tool to achieve global well-posedness, controlling $\omega$ in $L^1_{loc}(\RR_{+}; L^\infty)$. Unhappily, for 3D case, this criterion breaks down because   of the {\it stretching term} which  is one of the main sources of difficulties in the well-posedness theory of 3D Navier Stokes. There
are partial results in the case of the so-called axisymmetric flows without swirl. We say that a vector field $v$ is axisymmetric if it has the form:

\begin{equation}\label{Eq:1}
v(t, x)=v^{r}(t, r, z)\vec{e}_{r}+v^{z}(t, r, z)\vec{e}_{z},
\end{equation}
where, $(r,\theta,z)$ refers to the cylindrical coordinates in $\RR^3$, defined by setting $x=(r\cos\theta,r\sin\theta,z)$ with $0\le\theta<2\pi$ and the triplet  $(\vec{e}_{r}, \vec{e}_{\theta}, \vec{e}_{z})$ indicates the usual frame of unit vectors in the radial, toroidal and vertical directions with the notation
\begin{equation*}
\vec{e}_r=\Big(\frac{x_1}{r},\frac{x_2}{r},0\Big),\quad \vec{e}_{\theta}=\Big(-\frac{x_2}{r},\frac{x_1}{r},0\Big),\quad \vec{e}_{z}=(0,0,1).
\end{equation*}
The formula \eqref{Eq:1} allows us to reduce the expression of the vorticity to a scalar function \mbox{$\omega\triangleq\omega_{\theta}\vec{e}_{\theta}$}, with \mbox{$\omega_{\theta}=\partial_{z}v^r-\partial_{r}v^z$}. In this case, the stretching term $\omega\cdot\nabla v$ closes to $\frac{v^r}{r}\omega$, in particular the time evolution of $\omega_{\theta}$ reads as follows  
\begin{equation}\label{omega(theta)}
\partial_{t}\omega_{\theta}+(v\cdot\nabla)\omega_{\theta}-\mu\Delta\omega_{\theta}+\mu\frac{\omega_{\theta}}{r^2}=\frac{v^r}{r}\omega_{\theta},
\end{equation}
with the notation $v\cdot\nabla=v^{r}\partial_{r}+v^{z}\partial_{z}$ and $\Delta=\partial_{r}^{2}+\frac1r\partial_{r}+\partial_{z}^{2}$. Moreover, it easy to check that the quantity $\Pi=\frac{\omega_{\theta}}{r}$  obeys the following transport--diffusion equation
\begin{equation}\label{Pi}
\partial_{t}\Pi+v\cdot\nabla \Pi-\mu\Big(\Delta +\frac{2}{r}\partial_r\Big)\Pi=0,\quad\Pi_{| t=0}=\Pi_0.  
\end{equation}
The first remark, however, we are able to distinguish is that $\Pi$ satisfies almost the same equation as $\omega$ in 2D. Thus, the axisymmetric structure is considered in some sense as a reduction of dimension. Second, an $L^p-$estimate for \eqref{Pi} gives 
\begin{equation}\label{Pi-L-p}
\|\Pi(t)\|_{L^p}\le \|\Pi_0\|_{L^p}, \quad p\in[1,\infty]. 
\end{equation}
This estimate offered a good framework to M. Ukhoviskii and V. Yudovich \cite{uy}, and independently O. Ladyzhenskaya \cite{Ladyzhenskaya} to claim that \eqref{NS(mu)} admits a global and unique solution as long as $v_0\in H^1$ and $\omega_0,\frac{\omega_0}{r}\in L^2\cap L^\infty$. This latter weakened next by S. Leonardi, J. M\`alek, J. Nec\u as and M. Pokorn\'y in \cite{LMNP} once $v_0\in H^2$ and relaxed recently in \cite{a} by H. Abidi  for $v_0\in H^{\frac{1}{2}}$. Further, in terms of the vorticity for rough initial data, the system \eqref{NS(mu)} has tackled by many authors. It sould be emphisized that in \cite{Feng-Sverak}, H. Feng and V. Sver\'ak recently settled a global result in time for \eqref{NS(mu)}, in a particular case that the initial vorticity $\omega_0$ is supported on a circle. This result was developed lately by Th. Gallay and V. Sver\'ak \cite{Gallay-Sverak} in the more general case. Especially they proved that \eqref{NS(mu)} is globally well-posed in time whenever $\omega_0\in L^1(\Omega)$ with $\Omega=\{(r,z)\in\RR^2: r>0\}$ is a half-space equipped with the measure $drdz$. Such result was enlarged by the same authors to this case where $\omega_{0}$ is only a finite measure with small atomic part (see \cite{Gallay-Sverak}).  More recently, the same authors proved in \cite{Gallay-Sverak1} that the atomic part of the initial measure can be taken arbitrary large for stemming from circular vortex filaments, i.e $\omega_0= \gamma \delta_{\bar{r},\bar{z}}$, with $\gamma,\bar{r}>0$ and $z\in \mathbb{R}.$ However, the wellposedness in the general case of arbitrary measures is still open.
 
\hspace{0.5cm} As regards to 3D Boussinesq system \eqref{B(mu,kappa)}, the well-posedness subject-matter has considerably explored. In fact, Danchin and Paicu revisited in \cite{Danchin-Paicu} the solutions \`a la Leray and \`a la Fujita-Kato for \eqref{B(mu,kappa)} in the case where $\kappa=0$ and demonstrated that solutions are global and unique in time under smallness condition. In the same way, Abidi-Hmidi-Keraani \cite{Abidi-Hmidi-Keraani} asserted that \eqref{B(mu,kappa)} admits a unique global solution in axisymmetric setting  
as long as $v_0\in H^1(\RR^3),\;\Pi_0\in L^{2}(\RR^3), \rho_0\in L^2\cap L^\infty$ with $\supp\rho_{0}\cap(Oz)=\emptyset$ and $P_{z}(\supp\rho_0)$ is a compact set in $\RR^3$ to prohibit the violent singularity $\frac{\partial_r\rho}{r}$, with $P_{z}$ being the orthogonal projector over $(Oz)$. In the same fashion, this problem has been considered by Hmidi-Rousset in \cite{Hmidi-Rousset} for $\kappa\geqslant 0$. First, they declined the assumption on the support of the density. Second, they benefited from the coupling phenomena between the vorticity and the density by invoking a new unknown $\Gamma=\Pi-\frac{\rho}{2}$ which satisfies the equation
\begin{equation*}
\partial_t\Gamma+v\cdot\nabla\Gamma-\Big(\Delta+\frac{2}{r}\partial_r\Big)\Gamma=0,\quad\Gamma_{|t=0}=\Gamma_0.
\end{equation*}
It is easily seen that $\Gamma$ cheeks the same role as $\Pi$ for the Navier-Stokes system \eqref{NS(mu)}. The main interest of $\Gamma$ is to derive by a simple way a priori estimates for $\Pi$. More recently, the second an third authors conducted a new result in the sense that exploited the axisymmetric structure on the velocity  and the crititical regularity \`a la Fujita-Kato to assert that \eqref{B(mu,kappa)}, for $\kappa=0$, possesses a unique global solution as long as and $(v_0,\rho_0)\in H^{\frac12}\cap\dot B^{0}_{3,1}\times L^2\cap \dot B^{0}_{3,1} $. Finally, in the case $\kappa=\mu>0$, the authors succeed lately in \cite{Hanachi-Houamed-Zerguine} to perform a new result of global well-posedeness for \eqref{B(mu,kappa)} in the setting that $(\omega_0,\rho_0)$ is axisymmetric and belonging to the critical Lebesgue spaces $L^1(\Omega)\times L^1(\RR^3)$ in the spirit of \cite{Gallay-Sverak} concerning Navier-Stokes equations \eqref{NS(mu)}.

\subsection{Aims}The current paper occupies with a topic of the global well-posedness for the system \eqref{B(mu,kappa)} given under vorticity-density formulation\footnote{For simplicity, we take $\kappa=\mu=1$. However, we should mention that our arguments hold also for $\kappa=\mu>0.$}
\begin{equation}\label{VD}
\left\{ \begin{array}{ll} 
\partial_{t}\omega_{\theta}+v\cdot\nabla \omega_{\theta}-\frac{v^r}{r}\omega_{\theta}-\big(\Delta-\frac{1}{r^2}\big)\omega_{\theta}=-\partial_r\rho, & \\
\partial_{t}\rho+v\cdot\nabla \rho- \Delta\rho=0 & \\ (\omega_\theta,{\rho})_{| t=0}=({\omega}_0,{\rho}_0). \end{array} \right.
\end{equation}
We aim to extending our result latterly established in \cite{Hanachi-Houamed-Zerguine} to the larger class of initial data of measure type, i.e, for $(\omega_0,\rho_0)\in \mathscr{M}(\Omega)\times \mathscr{M}(\RR^3)$, where, $\mathscr{M}(\mathbb{X})$ denotes the set of finite Radon measures on $\mathbb{X}\in \{ \mathbb{R}^3,\Omega\}$ being such that
\begin{equation}
\|\mu\|_{\mathscr{M}(\mathbb{X})}\triangleq\sup_{\{\varphi\in C_{0}(\mathbb{X}), \|\varphi\|_{L^\infty(\mathbb{X})}\le1\}}|\langle\mu,\varphi\rangle|<\infty,
\end{equation}
with $\langle\cdot,\cdot\rangle$ symbolizes the pairing between $\mathscr{M}(\mathbb{X})$ and $C_{0}(\mathbb{X})$ which is defined by  \begin{equation*}
\langle\mu,\varphi\rangle\triangleq\int_{\mathbb{X}}\varphi(x)d\mu(x) . 
\end{equation*}

\hspace{0.5cm}To streamline the above aims, we fix some notations freely used in course of this paper. Let us start by endowing the half-space $\Omega=\{(r,z)\in\RR^2: r>0\}$ with the two-dimensional measure $drdz$, as opposed to 3D measure $rdrdz$ which seems to be convenient for axisymmetric topics. Thus, for $p\in[1,\infty]$, the Lebesgue space $L^p(\Omega)$ as the set of measurable functions $\omega_{\theta}:\Omega\rightarrow\RR$ such that the norm
\begin{equation*}
\|\omega_{\theta}\|_{L^p(\Omega)}=\left\{\begin{array}{ll}
\Big(\int_{\Omega}|\omega_{\theta}(r,z)|^p drdz\Big)^{\frac1p} & \textrm{if $p\in [1,\infty)$,}\\
\essup_{(r,z)\in\Omega}|\omega_{\theta}(r,z)| & \textrm{if $p=\infty$}
\end{array}
\right.
\end{equation*}
is finite. Also, we recall that $C_{0}(\Omega)$ (resp. $C_0(\RR^3)$) refers to the set of all continuous functions over $\Omega$ (resp. $\RR^3$) that vanish at infinity and on the boundary $\partial\Omega$. For $f\in L^1(\Omega)$ define the measure $\mu_{f}$ by $\langle\mu_f,\varphi\rangle=\int_{\Omega}f(x)\varphi(x)dx$, where $dx$ designates the Lebesgue measure. It can easily be seen that $\mu_f\in\mathscr{M}(\Omega)$, thus we deduce that $L^1(\Omega)$ can be identified as a closed subspace of $\mathscr{M}(\Omega)$ and $\|\mu_f\|_{\mathscr{M}(\Omega)}=\|f\|_{L^1(\Omega)}$. Furthermore, the space $\mathscr{M}(\Omega)$ is the tolopogical dual of $C_0(\Omega)$, so the Banach-Alaoglu theorem, insures that the unit ball in $\mathscr{M}(\Omega)$ is a sequentially compact set for the weak topology in the following sense: 
\begin{equation}\label{weak-converge}
\mu_n\rightharpoonup \mu,\;\;\mbox{if}\;\;\lim_{n\rightarrow\infty}\langle\mu_n,\varphi\rangle=\langle\mu,\varphi\rangle\quad\forall\varphi\in C_0(\Omega).
\end{equation}
Each $\mu\in\mathscr{M}(\Omega)$ can be decomposed in unique way as 
\begin{equation}\label{decomposition}
\mu=\mu_{ac}+\mu_{sc}+\mu_{pp}, \quad \mu_{ac}\perp\mu_{sc}\perp\mu_{pp}
\end{equation}
and
\begin{equation*}
\|\mu\|_{\mathscr{M}(\Omega)}=\|\mu_{ac}\|_{\mathscr{M}(\Omega)}+\|\mu_{sc}\|_{\mathscr{M}(\Omega)}+\|\mu_{pp}\|_{\mathscr{M}(\Omega)}.
\end{equation*}

where, in the sequel we denote by:\\ 
$\bullet$ $\mu_{ac}$ is a measure which is absolutely continuous with respect to Lebesgue measure, that is $\dfrac{d\mu_{ac}}{dx}=f$ for some $f\in L^1(\Omega)$,\\ 
$\bullet$ $\mu_{sc}$ is a singular continuous measure which has no atom but is supported on a set of zero Lebesgue measure.\\
$\bullet$ $\mu_{pp}$ is an atomic measure, $\mu_{pp}=\sum_{n\ge1}\lambda_{n}\delta_{a_n},\;(\lambda_n)\subset\RR,\;(a_n)\subset\Omega$, with $\delta_{a_n}$ stands to be the Dirac measure supported at  $a_n\in\Omega$.

\subsection{Statement of the main results}This subsection addresses to state the main results of this paper and thrash out the headlines of their proofs.  

\hspace{0.5cm}To make our presentation more convenient, we state the following result as a reference throughout this paper, and its demonstration in explicit way can be found in \cite{Hanachi-Houamed-Zerguine}.  
\begin{Theo}\label{First-Th.} Let $(\omega_{0},\rho_{0})\in L^1(\Omega)\times L^1 (\RR^3)$ be axisymmetric initial data, then the Boussinesq system \eqref{VD}   admits a unique global mild axisymmetric solution $(\omega,\rho)$ being such that
\begin{equation}\label{first-result}
(\omega_{\theta}, r\rho) \in \big(  C^0\big([0,\infty);  L^1(\Omega)\big)\cap C^0\big((0,\infty); L^\infty(\Omega)\big) \big)^2.
\end{equation}
\begin{equation}\label{first-result2}
 \rho \in C^0\big([0,\infty);  L^1(\RR^3)\big)\cap C^0\big((0,\infty); L^\infty(\RR^3)\big).
\end{equation}
Furthermore, for every $p\in [1,\infty)$, there exists some constant $K_p(D_0)>0$, for which, and for all $t>0$ the following statements hold 
\begin{equation}\label{boundedness}
\|(\omega_{\theta}(t),r\rho(t))\|_{L^p(\Omega)\times L^p(\Omega)}\leq t^{-(1-\frac{1}{p})}  K_p(D_0) .
\end{equation}
\begin{equation}\label{boundedness2}
\| \rho(t) \|_{L^p(\RR^3) }\leq t^{-\frac{3}{2} (1-\frac{1}{p})}  K_p(D_0),
\end{equation}
where
\begin{equation*}
D_0 \triangleq\|(\omega_0,\rho_0)\|_{L^1(\Omega)\times L^1(\RR^3)}.
\end{equation*}
\end{Theo}
 
\hspace{0.5cm}The main contribution of this paper is dedicated to extending the result of Theorem \ref{First-Th.} to more general case of initial data, that is to say the initial data belonging to the class of finite measure over $\Omega\times\RR^3$. Specifically, we shall prove the following theorem.
\begin{Theo}\label{thm-mesure-2}
There exists a non negative constant $\EE>0$ such that the following hold. Let $(\omega_0,\rho_{0})\in\mathscr{M}(\Omega)\times \mathscr{M}(\RR^3)$ with $\rho_0$ being axisymmetric in the sense of Definition \ref{def:axi:measure} and
  \begin{equation}\label{condition-theorem-2}
 \|\omega_{0,pp}\|_{\mathscr{M}(\Omega)}+ \|\rho_{0,pp}\|_{\mathscr{M}(\RR^3)}   \le \EE,
\end{equation}  
then, the Boussinesq system \eqref{VD}  admits a unique global mild axisymmetric solution $(\omega_{\theta},\rho)$ such that
\begin{equation*}
(\omega_{\theta},\rho)\in C^0\big((0,\infty);L^1(\Omega)\cap L^\infty(\Omega)\big)\times C^0\big((0,\infty);L^1(\RR^3) \cap L^\infty(\RR^3)\big),
\end{equation*}
\begin{equation*}
  r\rho \in C^0\big((0,\infty);L^1(\Omega)\cap L^\infty(\Omega)\big).
\end{equation*} 
Furthermore, for every $p\in (1,\infty]$, we have
\begin{equation*}
 \limsup_{t\rightarrow 0}\; t^{\frac{3}{2}(1-\frac{1}{p})}\| \rho(t) \|_{L^{p} (\RR^3)}\le C\EE  ,\quad \limsup_{t\rightarrow 0}\; t^{1-\frac{1}{p}}\|(\omega_{\theta}(t),r\rho(t))\|_{L^{p}(\Omega)\times L^{p}(\Omega)}\le C\EE 
\end{equation*}
and 
\begin{equation*}
  \limsup_{t\rightarrow 0}\;  \|(\omega_{\theta}(t),\rho(t))\|_{L^{1}(\Omega)\times L^{1}(\mathbb{R}^3)}<\infty.
\end{equation*}
Moreover, $(\omega_{\theta}(t),\rho(t))$ converges to  $(\omega_0,\rho_0) $, as $  t\rightarrow 0 $, in the sense of distribution.
\end{Theo}
Few remarks are in order. 
\begin{Rema} 
Observe that Theorem \ref{thm-mesure-2} covers a class of initial data which is considerably larger than the one treated by Theorem \ref{First-Th.}. However, the smallness condition for the atomic parts is crucial in our arguments to guarantee the existence and the uniqueness. Nevertheless, we should point out that it is probably possible to construct global solutions for arbitrary large initial data (see \cite{Feng-Sverak} for a precise result of that in the case of the vortex rings). Such  result is based on smoothing out the initial data, hence, the uniqueness is to be dealt with separately to the existence part, which, on the other hand, stands to be open in general for the time being even for the case the Navier-Stokes equations.
\end{Rema} 
\begin{Rema} The serious drawback that arises in Theorem \ref{thm-mesure-2}  is how to give a rigorous and suitable sense to the initial data associated with the quantity $r\rho$ if the initial density $\rho_0$ is a finite measure. More precisely, it is important to make a choice that does not perturb the weak continuity of the solution near $t=0$. To get around this, we should understand and give some general notions on the axisymmetric measures in $\RR^3 $ (see, the next section). In broad terms, one should notice that, with Theorem \ref{First-Th.} on the hand, the most challenging part in the proof of Theorem \ref{thm-mesure-2} is the understanding of the solution near $t=0$. Indeed, after some $t_0>0$, the solution becomes more regular and, hence, the arguments of Theorem \ref{First-Th.} would apply to garantee  estimates alike \eqref{boundedness} and \eqref{boundedness2}, globally in time. In other words, to be more precise on the main novel part in Theorem \ref{thm-mesure-2}, we refer to Theorem \ref{Th2:general version} that we shall prove at the end of this paper.
\end{Rema}
\begin{Rema}
Remark that the case $p=\infty$ is missing in the statement of Theorem \ref{First-Th.}. However, as we shall see later on, the estimates of Theorem \ref{First-Th.} hold also in this case. To justify our claim, we will outline the idea to get the $L^\infty$-estimate in the proof of Proposition \ref{properties-decay-Lp} below. 
\end{Rema}

{\bf Structure of the paper.} We discuss concisely the steps of the proof of Theorem \ref{thm-mesure-2} and the structure of the paper. The local well-posedness will be done via the classical fixed point formalism in an adequate functional spaces as in the proof of Theorem \ref{Second-Th.}. But before doing so, as noticed in the proof of Theorem \ref{First-Th.} \cite{Hanachi-Houamed-Zerguine}, since the quantity $r\rho$ will play a significant role in our analysis, we have to give a suitable sense to the limit of $  r\rho(t)$, as $t$ tends to $0$, in the case of initial measure-type density $\rho_0.$ For a better understanding of this limit, we state in Section \ref{section prelim} a succinct about the measure theory, in particular the push forward of a measure by a measurable function in the general case. This can be also considered as a preamble to introduce the concept of axisymmetric measures. In the second part of Section \ref{section prelim}, we shall recall some nice properties of the semi--groups associated to the system in study. Thereafter, Section \ref{section proof-main} contains three parts: In the first one, we prove a result which can be seen as an intermediate case between  Theorem \ref{First-Th.} and Theorem \ref{thm-mesure-2}  where we assume that $(\omega_0,\rho_0)\in\mathscr{M}(\Omega)\times L^1(\RR^3)$. The details of the proof we provide in this part should help to simplify the presentation of the proof of  Theorem \ref{thm-mesure-2}. Then, in the second part, we shall prove a general version of the local well-posedness (see Theorem \ref{Th2:general version}) which implies the local results of Theorem \ref{thm-mesure-2}. Finally, in the last part of Section \ref{section proof-main}, we outline the idea that allows to globalize the local results we prove in the first two parts.

\section{Setup and preliminary results}\label{section prelim}
In this preparatory section we gather the basic ingredients freely explored during this work. We begin with a self contained abstract on some notions from measures theory. Then, we recall some estimates of the heat semi-group   in Lebesgue spaces.
\subsection{Results on measure theory}\label{sub-section: measures}
We embark by the measure tool, where we state the notion of the  push-forward measure, some nice properties and we state a new concept about the axisymmetric measures and we give two fundamental examples. Overall, we claim the action of the axisymmetric measure  on the heat semi-group. We end this section by an important properties concerning the restriction of any axisymmetric measure on $\Omega$ in order to define the quantity $r\rho.$
\begin{defi}\label{pushforward-def}
Let $(X_1,\Sigma_1)$ and $(X_2,\Sigma_2)$ be two measurable spaces and $\mu$ be a positive measure on $(X_1,\Sigma_1)$. Let $F$ be a measurable mapping from $X_1$ into $X_2$. The  push-forward measure of $\mu$ by $F$, denoted $F_\star\mu$, is defined as 
\begin{align*}
F_\star\mu: &\Sigma_2 \rightarrow [0,\infty]\\
& B \longmapsto  F_\star\mu (B) \triangleq \mu (F^{-1}(B)).
\end{align*}
\end{defi}
The main feature of the above definition is the fact that it  is useful in the following generalization formula of change of variable (see Sections 3.6--3.7 in \cite{Bogachev})
\begin{Theo}\label{change of variables formula}
A measurable function $g$ on $X_2$ is integrable with respect to the push-forward measure $F_\star\mu$ if and only if the composition $g\circ F$ is integrable with respect to the measure $\mu$. As well, we have
\begin{equation}\label{G-ch-v}
\int_{X_2} gd(F_\star\mu)=\int_{X_1}g\circ F d\mu.
\end{equation}
\end{Theo}

The typical example that concerns the axisymmetric structure given as follows. For $\alpha\in[0,2\pi)$, define $\mathfrak{R}_{\alpha}:\RR^3\rightarrow\RR^3$ by $x\mapsto\mathfrak{R}_{\alpha}x$, with 
\begin{equation}\label{Rot-alpha}
\mathfrak{R}_{\alpha}\triangleq \left(
\begin{array}{ccc}
\cos \alpha& -\sin\alpha& 0 \\
\sin\alpha & \cos\alpha &  0 \\
0 & 0 & 1
\end{array}
\right),\quad x^{t}=(x_1,x_2,x_3).
\end{equation}
An elementary calculus shows that $\mathfrak{R}_{\alpha}$ is an orthogonal $3\times 3$ matrix, with $\mathfrak{R}_{\alpha}^{-1}=\mathfrak{R}_{-\alpha}$. In addition, by exploring Definition \ref{pushforward-def}, for $\mu\in \mathscr{M}(\RR^3) $, we can define its push-forward measure $\mathfrak{R}_{\alpha}\mu$ as an element of $\mathscr{M}(\RR^3) $. Moreover, Theorem \ref{change of variables formula} provides us the following identity
\begin{equation*}
\langle \mathfrak{R}_{\alpha}\mu,\varphi\rangle=\langle\mu,\varphi \circ\mathfrak{R}_{-\alpha}\rangle,\quad\forall \varphi\in C^{0}(\RR^3).
\end{equation*} 
Based on that, we state the following definition.
\begin{defi}[\textbf{Axisymmetric measure}]\label{def:axi:measure}
A Radon measure  $\mu\in\mathscr{M}(\RR^3)$  is  said to be axisymmetric if and only if it is stable by the push-forward mapping $\mathfrak{
R}_{\alpha}$, for all $\alpha\in [0,2\pi)$. i.e, if 
\begin{equation}\label{axi-identity}
\mathfrak{R}_{\alpha} \mu = \mu, \quad \forall \alpha \in [0,2\pi).
\end{equation}
\end{defi}
\begin{Rema}
The above definition says   that the measure $\mu$ is   axisymmetric if and only if it is stable by the push-forward mapping $\mathfrak{R}_{\alpha} $, for all $\alpha\in [0,2\pi)$. We can, moreover, check that this definition is equivalent to 
\begin{equation*}
\mu(B)=\frac{1}{2\pi}\int_{0}^{2\pi}\mathfrak{R}_{\alpha}\mu( B)d\alpha,\quad \forall B\in\mathscr{B}(\RR^3).
\end{equation*}
Or, also equivalently  
\begin{equation*}
d\mu=\frac{1}{2\pi}\int_{0}^{2\pi}d(\mathfrak{R}_{\alpha} \mu)d\alpha.
\end{equation*}

\end{Rema}

To illustrate the above definition we state here two typical examples of an axisymmetric measure. 
\begin{itemize}
\item \textit{Absolutely continuous axisymmetric measures }: If $\mu\in \mathscr{M}_{ac}(\RR^3)$, then there exists an integrable function \mbox{$f_\mu \in L^1(\RR^3)$} such that 
$$\mu(B)= \int_B f_\mu(x) dx, \quad \forall B\subset \RR^3. $$
In this case, we can check that $\mu$ is axisymmetric in the sense of Definition \ref{def:axi:measure} if and only if $f_\mu$ is an axisymmetric function in the classical sense.\\
 \item \textit{Atomic axisymmetric measures }: Let $a\in \RR^3$ and $\mu= \delta_a$,  one can check that $\mu$ is axisymmetric in the sense of Definition \ref{def:axi:measure} if and only if $\{a\}$ is stable by rotation around the $(oz)$ axis. In other words, $\delta_a$ is axisymmetric if and only if  $a\in (oz).$ More generally, if $A\subset \mathbb{R}^3$, then $\mu=\delta_A$ is axisymmetric if and only if, there exists $(r,z)\in \mathbb{R}^+\times \mathbb{R}$ such that 
 \begin{equation*}
 A= \bigcup_{\theta \in [0,2\pi]} \big\{(r\cos \theta, r\sin \theta, z)  \big\}.
 \end{equation*}
\end{itemize} 
As  a nice consequence of the above ingredients we have the following elementary result.
\begin{prop}\label{axi:sol:heat}
Let  $\mu\in\mathscr{M}(\RR^3)$ be an axisymmetric measure, then the function $x\mapsto e^{t\Delta}\mu(x)$ is axisymmetric.
\end{prop}
\begin{proof}
We need   to show that, for all $\alpha\in [0,2\pi)$, there holds
\begin{equation*}
e^{t\Delta}\mu(x)= e^{t\Delta}\mu (\mathfrak{R}_{\alpha} x ) .
\end{equation*}
To do so, we write 
\begin{eqnarray*}
 e^{t\Delta}\mu(\mathfrak{R}_{\alpha}(x)) &=&\frac{1}{(4\pi t)^\frac32} \int_{\RR^3}e^{-\frac{  |\mathfrak{R}_{\alpha}x-y|^2}{4t}}d\mu(y) \\
&=&\frac{1}{(4\pi t)^\frac32} \int_{\RR^3}e^{- |\mathfrak{R}_{\alpha}|^2\frac{|x-\mathfrak{R}_{-\alpha}y|^2}{ 4t} }d\mu(y) .
\end{eqnarray*}
Now, exploring the fact that $|\mathfrak{R}_{\alpha}|^2=1$ and by taking $g_{t,x}( y)=\frac{1}{(4\pi t)^\frac32}e^{-\frac{|x- y|^2}{ 4t} }$, then Theorem \ref{change of variables formula} implies
\begin{eqnarray*}
 e^{t\Delta}\mu(\mathfrak{R}_{\alpha}(x))d\alpha&=& \int_{\RR^3}g_{t,x}\circ \mathfrak{R}_{-\alpha}(y )d  \mu(y)  \\
&=&  \int_{\RR^3}g_{t,x} (y )d(\mathfrak{R}_{\alpha}\mu(y)) 
 \end{eqnarray*}
 Now, because $\mu$ is an axisymmetric measure, we infer that for all $\alpha\in [0,2\pi)$
\begin{eqnarray*}
 \int_{\RR^3}g_{t,x} (y )d(\mathfrak{R}_{\alpha}\mu(y))    &=&\int_{\RR^3}g_{t,x}(y)d\mu(y) .
\end{eqnarray*}
 This completes the proof.
\end{proof}
\begin{Rema}
For any PDE in general, and for the Boussinesq system in particular, if we are looking for special solutions, then,   we need to impose some kind of suitable compatibility condition for the initial data that fits well with the desired structure of the solutions.   Definition \ref{def:axi:measure} plays exactly that role here. In other words,  it provides the requirement of the structure of the initial data that allows to study the existence of axisymmetric solutions. Proposition \ref{axi:sol:heat} is then a typical example of the propagation of this special geometric structure of  the initial data for all times $t\geqslant 0$. 
\end{Rema}
\begin{Rema}
Let us also mention that a similar result to Proposition \ref{axi:sol:heat} would apply to more general equations such as 
\begin{equation}\label{B(ff)}
 f= e^{t\Delta}\mu + B(f,f), 
\end{equation}
if $B$ is a bi-linear operator, preserving the axisymmetric structure\footnote{ In the case of the Boussinesq system, $B$ takes the form $B(\rho,\rho)= \int_0^t e^{(t-s)\Delta} (v\cdot \nabla \rho)(s) ds$, and $ v$ is related to $\rho$ through the Navier-Stokes equations.}, and if a contraction argument is applicable to the system \eqref{B(ff)} in some time-space Banach space $X_T$. Indeed, the proof of such   result relies on proving that the sequence
\begin{equation*}
\left\{\begin{array}{ll}
f_n= e^{t\Delta}\mu + B(f_{n-1},f_{n-1}),& n\geq 1,\\
f_0 = e^{t\Delta}\mu.
\end{array}\right.
\end{equation*}
converges to some axisymmetric limit $f$ in the space $X_T$. This can be done easily under the aforementioned assumptions on $B$ and on the space $X_T.$
\end{Rema}

\begin{lem}[\textbf{Action on test functions}]\label{lemma:action:axi}
Let $\mu$ be an axisymmetric measure, then, for any \mbox{$\varphi \in C^0(\RR^3)$}, we have that
\begin{equation*}
\left\langle \mu, \varphi \right\rangle = \left\langle \mu, \varphi_{\axi}\right\rangle,
\end{equation*}
where,  $\varphi_{\axi}$ is the axisymmetric part of $\varphi$ given by
\begin{equation}\label{phi-axi:def}
\varphi_{\axi}\triangleq \frac{1}{2\pi}\int_0^{2\pi}\varphi \circ \mathfrak{R}_\alpha d\alpha.
\end{equation}
\end{lem}
\begin{proof}
The definition of axisymmetric measure and the identity \ref{axi-identity} allow us to write
\begin{equation*}
\langle\mu,\varphi\rangle=\langle\mathfrak{R}_{-\alpha}\mu,\varphi\rangle, \quad\forall\alpha\in[0,2\pi).
\end{equation*}
Then, straightforward computation yield
\begin{eqnarray*}
\langle\mu,\varphi\rangle&=&\frac{1}{2\pi}\int_0^{2\pi}\langle\mathfrak{R}_{-\alpha}\mu,\varphi\rangle d\alpha\\
&=&\frac{1}{2\pi}\int_0^{2\pi}\langle\mu,\varphi\circ\mathfrak{R}_{\alpha}\rangle d\alpha\\
&=&\langle\mu,\frac{1}{2\pi}\int_0^{2\pi}\varphi\circ\mathfrak{R}_{\alpha} d\alpha\rangle\\
&=&\langle\mu,\varphi_{\axi}\rangle.
\end{eqnarray*}
The lemma is then proved.
\end{proof}
\begin{Rema}
We point out again that the function $\varphi_{\axi}$ is  indeed axisymmetric due to the following elementary computation, valid for all $\theta\in [0,2\pi)$
\begin{eqnarray*}
\varphi_{\axi}\circ\mathfrak{R}_\theta&=&\frac{1}{2\pi}\int_0^{2\pi}\varphi\circ\mathfrak{R}_{\alpha+\theta}d\alpha\\
&=&\frac{1}{2\pi}\int_\theta^{2\pi+\theta}\varphi\circ\mathfrak{R}_{\gamma}d\gamma\\
&=&\frac{1}{2\pi}\int_0^{2\pi}\varphi\circ\mathfrak{R}_{\gamma}d\gamma\\
&=& \varphi_{\axi}.
\end{eqnarray*}  
Lemma \ref{lemma:action:axi} says then that, when we deal with axisymmetric measures, we can restrict our test functions to be axisymmetric ones.
\end{Rema}
Now, let us define the function $\mathfrak{F}$ as the mapping from $\Omega\times [0,2\pi)$ into $\RR^3$, with $\Omega=(0,\infty) \times \RR$, given by 
\begin{equation}\label{F-r-z-theta}
\mathfrak{F}(r,z,\theta) \triangleq \mathfrak{R}_\theta \cdot (r,0,z)^t= (r\cos \theta, r\sin \theta,z).
\end{equation}

The following proposition will serve latter to prove our main Theorem \ref{thm-mesure-2}.
\begin{prop}\label{prop:mu-axi}
Let $\mu$ be an axisymmetric measure in $\mathscr{M}(\RR^3)$. Then, the mapping $\widetilde{\mu}$ defined on $C^0(\Omega)$ as 
\begin{equation}\label{Def-mu-tilde}
\left\{ \begin{array}{c}
\left\langle \widetilde{\mu}, \psi\right\rangle \triangleq   \displaystyle \int_{\RR^3} \phi_\psi d\mu,\qquad\forall\psi\in C^0(\Omega)\\~~\\
\phi_\psi(x,y,z) \triangleq \psi(\sqrt{x^2+y^2},z),
\end{array}\right.
\end{equation}
belongs to $\mathscr{M}(\Omega)$ and satisfies, for any function $\varphi \in C^0(\RR^3)$ and for all $\theta \in [0,2\pi)$
\begin{equation}\label{identity-axi-omega}
\int_\Omega \varphi_{\axi} \circ \mathfrak{F}(r,z,\theta) d\tilde{\mu}(r,z)=\int_\Omega \varphi_{\axi} \circ \mathfrak{F}(r,z,0) d\tilde{\mu}(r,z) =   \int_{\RR^3} \varphi d\mu,
\end{equation}
where, $\varphi_{\axi}$ is given by \eqref{phi-axi:def}.
Moreover, we have that
\begin{equation}\label{norm-axi-omega}
\|\widetilde{\mu} \|_{\mathscr{M}(\Omega)}=\| \mu  \|_{\mathscr{M}(\RR^3)}.
\end{equation}
Furthermore, the following holds. If $\mu= \mu_{ac}+ \mu_{pp}+ \mu_{sc}$ is the Lebesgue decomposition of $\mu$ with $\mu_{ac}, \mu_{pp}$ and $\mu_{sc}$ denote, respectively, the absolute continuous part of $\mu$, its atomic and its singular continuous part, then, $\widetilde{\mu}= \widetilde{\mu}_{ac}+ \widetilde{\mu}_{pp}+ \widetilde{\mu}_{sc}$ with
 $$ \|\widetilde{\mu}_{ac}\|_{\mathscr{M}(\Omega)} =  \|\mu_{ac}\|_{\mathscr{M}(\RR^3)} ,$$
 $$ \|\widetilde{\mu}_{pp}\|_{\mathscr{M}(\Omega)} =   \|\mu_{pp}\|_{\mathscr{M}(\RR^3)},$$
 $$ \|\widetilde{\mu}_{sc}\|_{\mathscr{M}(\Omega)} =   \|\mu_{sc}\|_{\mathscr{M}(\RR^3)}.$$
\end{prop}
\begin{Rema}
Remark that, in view of Lemma \ref{lemma:action:axi}, the equality  \eqref{identity-axi-omega} yields, for any axisymmetric function   $\varphi \in C^0(\RR^3)$ and for all $\theta \in [0,2\pi)$
\begin{equation}\label{identity-axi-omega:222}
\int_\Omega \varphi \circ \mathfrak{F}(r,z,\theta) d\tilde{\mu}(r,z)=\int_\Omega \varphi \circ \mathfrak{F}(r,z,0) d\tilde{\mu}(r,z) =   \int_{\RR^3} \varphi d\mu,
\end{equation} 
\end{Rema}
\begin{proof}
From the definition of $\widetilde{\mu}$, we can easily check that it belongs to $\mathscr{M}(\Omega)$. We shall then focus on the proof of \eqref{identity-axi-omega}. Remark first that  the fact that $\varphi_{\axi}$ is axisymmetric insures  
\begin{equation*}
\varphi_{\axi}\circ \mathfrak{F}(r,z,\theta) = \varphi_{\axi}\circ \mathfrak{F}(r,z,0)   ,\quad \forall (r,z,\theta)\in \Omega\times[0,2\pi).
\end{equation*} 
which is a direct consequence of the fact that
\begin{equation}\label{AAAAA}
\varphi_{\axi} (x,y,z) = \varphi_{\axi}(\sqrt{x^2+y^2},0,z) =\varphi_{\axi} \circ \mathfrak{F}( \sqrt{x^2+y^2},z,0)  ,\quad \forall (x,y,z)\in \RR^3.
\end{equation}  The first equality on the l.h.s of   \eqref{identity-axi-omega} follows then. For the second equality, we only have to use the definition of $\widetilde{\mu}$   together with \eqref{AAAAA} to infer that 
\begin{equation*}
\int_\Omega \varphi_{\axi} \circ \mathfrak{F}(r,z,0) d\tilde{\mu}(r,z) =    \int_{\RR^3} \varphi_{\axi} d\mu.
\end{equation*} 
Consequently, the equality on the r.h.s on \eqref{identity-axi-omega} follows by applying Lemma \ref{lemma:action:axi}.
 
 \vspace{2mm} Now, concerning the size of $\widetilde{\mu}$, we only outline the proof of \eqref{norm-axi-omega}, meanwhile, the proof of the estimates for the decomposed parts is straightforward (see also the two examples below). From the definition of $\widetilde{\mu}$, it is easy to check that 
$$\|\widetilde{\mu} \|_{\mathscr{M}(\Omega)}\leqslant \| \mu  \|_{\mathscr{M}(\RR^3)} .$$
On the other hand, \eqref{identity-axi-omega} provides the converse inequality
\begin{equation*}
\| \mu  \|_{\mathscr{M}(\RR^3)} \leqslant \|\widetilde{\mu} \|_{\mathscr{M}(\Omega)}.
\end{equation*}
This ends the proof of Proposition \ref{prop:mu-axi}
\end{proof} 
For the sake of clarity, we provide the following two typical examples:
\begin{itemize}
\item {\it An example in $\mathscr{M}_{ac}(\RR^3):$} If $\mu$ is an axisymmetric measure in $  \mathscr{M}_{ac}(\RR^3)$, then, its associated density $f_\mu$ is an axisymmetric function in $L^1(\RR^3)$. In this case, $\widetilde{\mu}$ is the measure in $\mathscr{M}(\Omega)$ associated to the $L^1(\Omega)$-density function $ f _{\widetilde{\mu}}$ give by
\begin{equation*}
f _{\widetilde{\mu}}(r,z) = 2\pi r\, f_\mu(r,0,z).
\end{equation*}
\item {\it An example in $ \mathscr{M}_{pp}(\RR^3)$:} We saw that $\mu=\delta_A$ is an axisymmetric measure if and only if $A$ is invariant by rotation around the axis $(oz)$ (i.e, $A$ is a circle with axis $(oz)$). In this case, we have  
\begin{eqnarray*}
\left\langle \widetilde{\mu},\psi\right\rangle &=& \left\langle \delta_A, \phi_\psi \right\rangle \\
&=& \sum_{ (a_1,a_2,a_3)\in A } \psi (\sqrt{a_1^2+a_2^2},a_3)\\
&=& \left\langle \delta_{\widetilde{A}}, \psi \right\rangle,
\end{eqnarray*}
where, $\widetilde{A}= \left\{(\sqrt{a_1^2+a_2^2} ,a_3):\; (a_1,a_2,a_3)\in A)   \right\}.$
In particular, if $A=(0,0,a)$ then  \mbox{$\widetilde{\mu}= \delta_{(0,a)}$} and more generally, for any $r\geqslant 0$, if $A_r= \cup_{\theta\in [0,2\pi)} \left\{(r\cos \theta, r \sin\theta,a) \right\}$, then $\widetilde{\mu} = \delta_{(r,a)}.$
\end{itemize}

\subsection{Semi-group estimates}\label{Semi-group estimates} In this subsection, we recall some technical results concerning the semi-groups appearing in the study of the Boussinesq system in question. For more details about these results, we refer the reader to \cite{Hanachi-Houamed-Zerguine,Gallay,Gallay-Sverak}. 

\hspace{0.5cm}In the sequel, we shall be using the following notations: For $i\in \{ 1,2\}$, we denote by  $(\Ss_i(t))_{t\geq0}$ the semi-groups defined as the propagators associated with the following two linear equations respectively 
\begin{equation}\label{S1}
\left\{\begin{array}{ll}
\partial_t f -\left( \Delta - \frac{1}{r^2}\right)f=0,\\
f_{|_{t=0}}=f_0.
\end{array}\right.
\end{equation}
\begin{equation}\label{S2}
\left\{\begin{array}{ll}
\partial_t f -  \Delta   f=0,\\
f_{|_{t=0}}=f_0.
\end{array}\right.
\end{equation}

The following propositions, proved with details in \cite{Hanachi-Houamed-Zerguine,Gallay-Sverak}, present some $L^p-L^q$ estimates of the semi--groups $(\Ss_i(t))_{t\geq0}$.
\begin{prop}\label{P,S1,S2} The family $((\Ss_1(t),\Ss_2(t))_{t\ge0}$ associated to \eqref{S1} and \eqref{S2}, respectively, is a strongly continuous semi--group of bounded linear operators in $L^p(\Omega)\times L^p(\Omega) $ for any $p\in[1,\infty]$. Furthermore, for $1\le p\le q\le\infty$ the following assertions hold.
\begin{enumerate}
\item[{\bf(i)}] For $(\omega_0,\rho_0)\in L^p(\Omega)\times L^p(\Omega)$, we have for every $t>0$
\begin{equation}\label{S(t)}
\|(\Ss_1(t)\omega_0,\Ss_2(t)\rho_0)\|_{L^q(\Omega)\times L^q(\Omega)}\le\frac{C}{t^{\frac{1}{p}-\frac{1}{q}}}\|(\omega_0,\rho_0)\|_{L^p(\Omega)\times L^p(\Omega)}.
\end{equation}
\item[{\bf(ii)}] For $f=(f^r,f^z)\in L^p(\Omega)\times L^p(\Omega)$, we have for every $t>0$
\begin{equation}\label{S(t)div}
\|\Ss_1(t)\Div_{\star}f\|_{L^q(\Omega)}\le\frac{C}{t^{\frac{1}{2}+\frac{1}{p}-\frac{1}{q}}}\|f\|_{L^p(\Omega)}.
\end{equation}
\item[{\bf(iii)}] For $f=(f^r,f^z)\in L^p(\Omega)\times L^p(\Omega)$, we have every $t>0$
\begin{equation}\label{S2(t)div}
\|\Ss_2(t)\Div f\|_{L^q(\Omega)}\le\frac{C}{t^{\frac{1}{2}+\frac{1}{p}-\frac{1}{q}}}\|f\|_{L^p(\Omega)}.
\end{equation}
\end{enumerate}
Here, $\Div_{\star}f=\partial_r f^r+\partial_z f^z$ (resp.  $\Div f=\partial_r f^r+\partial_z f^z+\frac{f^r}{r}$) stands to be the divergence operator  defined over $\RR^2$ (resp. the divergence operator defined over $\RR^3$ in the axisymmetric case).
\end{prop}
Next, we recall the following weighted estimates from \cite{Hanachi-Houamed-Zerguine} in the spirit of Proposition 3.5 in \cite{Gallay-Sverak}.
\begin{prop}\label{Prop-W-E-1} Let $ 1\leq p \leq q \leq \infty$, $i\in \{ 1,2\}$ and $(\alpha,\beta)\in[-1,2]$, with $\alpha\leq \beta$. Assume that $r^{\beta}f\in L^p(\Omega)$, then
\begin{equation}\label{Eq:1-W-est-S1}
\|r^{\alpha}\Ss_i(t) f\|_{L^q(\Omega)}\le \frac{C}{t^{\frac1p-\frac1q+\frac{(\beta-\alpha)}{2}}}\|r^{\beta}f\|_{L^p(\Omega)}.
\end{equation}
In addition, if $(\alpha,\beta)\in[-1,1]$, $\alpha\leq \beta$ and $r^{\beta}f\in L^p(\Omega)$, then
\begin{equation}\label{Eq:2-W-est-S1}
\|r^{\alpha}\Ss_i(t)\Div_{\star} f \|_{L^q(\Omega)}\le \frac{C}{t^{\frac12+\frac1p-\frac1q+\frac{(\beta-\alpha)}{2}}}\|r^{\beta}f\|_{L^p(\Omega)}.
\end{equation}
\end{prop}
 The following proposition states some $L^p-L^q$ estimates in the case of   initial data in $\mathcal{M}(\Omega)$. The proof can be found in \cite{Gallay-Sverak}.
\begin{prop}\label{prop mu-lamda} For any $\mu\in\mathscr{M}(\Omega)$, we have 
\begin{equation}\label{Est1-S(t)}
\sup_{t>0}\; t^{1-\frac1q}\|\Ss_1(t)\mu\|_{L^q(\Omega)}\le C\|\mu\|_{\mathscr{M}(\Omega)},\quad q\in[1,\infty]
\end{equation}
and
\begin{equation}\label{Est2-S(t)}
L_{q}(\mu)\triangleq \limsup_{t\uparrow 0}\; t^{1-\frac1q}\|\Ss_1(t)\mu\|_{L^q(\Omega)}\le C\|\mu_{pp}\|_{\mathscr{M}(\Omega)},\quad q\in(1,\infty],
\end{equation}
where, $\mu_{pp}$ is the atomic part of $\mu.$
\end{prop} 
Finally, in the spirit of the previous propositions, we state a quite similar estimates for the semi--group $\Ss_2(t)$ in the space $\RR^3$ instead of $\Omega$. The proof is similar to the proof of Proposition \ref{prop mu-lamda}
\begin{prop}\label{3D-heat kernel} Let $ 1\leq p \leq q \leq \infty$. Assume that $ f\in L^p(\RR^3)$, then
\begin{equation}\label{Eq-3D-heat kernel}
\| \Ss_2(t) f\|_{L^q(\RR^3)}\le \frac{C}{t^{\frac{3}{2}(\frac{1}{p}- \frac{1}{q}) } }\| f\|_{L^p(\RR^3)}.
\end{equation}
Moreover, if $f= \mu \in \mathscr{M}(\RR^3)$, then the above estimate holds by taking $p=1$ and by replacing $\norm f_{L^1(\RR^3)}$ by $\norm f_{\mathscr{M}(\RR^3)}$. In addition of that, the following assertion holds
\begin{equation}\label{Eq-3D-heat kernel-measure}
\widetilde{L}_q(\mu) \triangleq \limsup_{t\rightarrow 0}\; t^{\frac{3}{2}(1- \frac{1}{q}) }\| \Ss_2(t) \mu\|_{L^q(\RR^3)}\le  C\| \mu_{pp}\|_{\mathscr{M}(\RR^3) }, \quad \forall q\neq 1,
\end{equation}
where, $\mu_{pp}$ is the atomic part of $\mu.$
\end{prop}


\section{Proof of the main results}\label{data-measure}\label{section proof-main}
\subsection{Median case: Only the initial vorticity is a finite measure} Before stating the proof of the main result, we embark this section by a particular result concerning the global well-posedness topic for \eqref{VD} in the case where the initial density is Lebesgue-integrable and the initial vorticity is a finite measure. The arguments of the proof for this result will be  considered as the platform to proving the Theorem \ref{thm-mesure-2}. 
    
\begin{Theo}\label{Second-Th.}There exist non negative constants $\EE $ and $C$ such that the following hold. Let  $(\omega_0,\rho_{0})\in\mathscr{M}(\Omega)\times L^1(\RR^3)$ with $\rho_0$ axisymmetric and  $\|\omega_{0,pp}\|_{\mathscr{M}(\Omega)}\le\EE$, then, the Boussinesq system \eqref{VD}  admits a unique global axisymmetric mild solution $(\omega_{\theta},\rho)$ satisfying
\begin{equation*}
(\omega_{\theta},\rho)\in C^0\big((0,\infty);L^1(\Omega)\cap L^\infty(\Omega)\big)\times C^0\big([0,\infty);L^1(\RR^3)\big)\cap C^0\big((0,\infty);L^\infty(\RR^3)\big),
\end{equation*}
\begin{equation*}
  r\rho \in C^0\big([0,\infty);L^1(\Omega)\big)\cap C^0\big(L^\infty(0,\infty);L^\infty(\Omega)\big).
\end{equation*}  
Furthermore, for every $p\in (1,\infty]$, we have
\begin{equation*}
 \limsup_{t\rightarrow 0}\; t^{\frac{3}{2}(1-\frac{1}{p})}\| \rho(t) \|_{L^{p} (\RR^3)}=\limsup_{t\rightarrow 0}\; t^{1-\frac{1}{p}}\|r\rho(t)\|_{  L^{p}(\Omega)}=0 
 \end{equation*}
 and
  \begin{equation*}
 \limsup_{t\rightarrow 0}\; t^{1-\frac{1}{p}}\|\omega_{\theta}(t) \|_{L^{p}(\Omega)}\le C\EE .
\end{equation*}
Moreover, we have
\begin{equation*}
  \limsup_{t\rightarrow 0}\;  \|\omega_{\theta}(t)\|_{L^{1}(\Omega) }<\infty, \quad \lim_{t\rightarrow 0}\|\rho(t)-\rho_0 \|_{L^1(\RR^3)} =0
\end{equation*}
and $\omega_{\theta}(t) \rightharpoonup \omega_0 $ as $  t\rightarrow 0. $
\end{Theo}


 The proof of Theorem \ref{Second-Th.} will be done in four steps. We begin with the proof of the local well-posedness   for the integral equations \eqref{int-equation} below (Proposition \ref{existence Locale-mu}), where, we cover the limits stated in Theorem \ref{Second-Th.} for $p=\frac{4}{3}$. Then, we provide a self contained proof of the remaining cases of $p$ by a bootstrap argument (Proposition \ref{properties-decay-Lp}). Thereafter, we establish the contunuity in time and the convergence to the initial data in Proposition \ref{Continuity of the solution}. Finally, the globalization the local solution we construct in step one is postponed to the end of the next section.
 
 All in all, along the proof of the three incoming propositions, we will only highlight the big lines of the proof since the idea is the pretty much similar to our previous work  \cite{Hanachi-Houamed-Zerguine}. Nevertheless, we shall provide the details of the crucial new technical issues. \\
 
 First, note that the Boussinesq system \eqref{VD} can be written   in the following form 
\begin{equation}\label{nonlin-prob}
\left\{ \begin{array}{ll} 
\partial_{t}\omega_{\theta}+\text{div}_{\star}(v\omega_{\theta})-\Big(\partial^{2}_{r}+\partial^{2}_{z}+\frac{1}{r}\partial_{r}-\frac{1}{r^2}\Big)\omega_{\theta}=-\partial_r\rho & \\
\partial_{t}\rho+\Div(v\rho)-\Delta \rho=0 &  \\ (\omega_\theta,{\rho})_{| t=0}=({\omega}_0,{\rho}_0).  \end{array} \right.
\end{equation}
Hence, according to the result of \cite{Hanachi-Houamed-Zerguine} the direct treatment of the local well-posedness topic for \eqref{nonlin-prob} in the spirit of \cite{Gallay-Sverak} for initial data $(\omega_{0},\rho_0)$ in the critical space  requires  the introduction of a new quantity $\widetilde{\rho} \triangleq r\rho$. The outcome system of this new unknown is given by the following parabolic equation
\begin{equation}\label{tilde-rho}
\partial_{t}\widetilde{\rho}+\Div_{\star}(v\widetilde{\rho})-\Big(\partial^{2}_{r}+\partial^{2}_{z}+\frac{1}{r}\partial_{r}-\frac{1}{r^2}\Big)\widetilde{\rho}=-2\partial_r\rho.
\end{equation} Thus, we shall consider the following system
\begin{equation}\label{int-equation}
\left\{\begin{array}{ll}
\omega_{\theta}(t)=\Ss_1(t)\omega_{0}-\int_{0}^{t}\Ss_1(t-\tau)\Div_{\star}\big(v(\tau)\omega_{\theta}(\tau)\big)d\tau-\int_{0}^{t}\Ss_1(t-\tau)\partial_{r}\rho(\tau) d\tau &\\~~\\
\widetilde{\rho}(t)=\Ss_1(t)\widetilde{\rho}_{0}-\int_{0}^{t}\Ss_1(t-\tau)\Div_{\star}\big(v(\tau)\widetilde{\rho}(\tau)\big)d\tau-2\int_{0}^{t}\Ss_1(t-\tau)\partial_{r}\rho(\tau) d\tau &\\~~\\
\rho(t)=\Ss_2(t)\rho_0-\int_{0}^{t}\Ss_2(t-\tau)\Div\big(v(\tau)\rho(\tau)\big)d\tau.
\end{array}
\right.
\end{equation}
where $\widetilde{\rho}_{0}=r\rho_0.$  In order to study the above system, we use the Banach spaces. 
\begin{equation*}
X_{T}=\Big\{f\in C^{0}\big((0,T] ,L^{4/3}(\Omega)\big): \Vert f\Vert_{X_{T}}<\infty \Big\},
\end{equation*}
\begin{equation*}
Z_{T}=\Big\{h\in C^{0}\big((0,T] ,L^{4/3}(\RR^3)\big): \Vert h\Vert_{Z_{T}}<\infty \Big\},
\end{equation*}
equipped with the following norms 
\begin{equation*}
\Vert f\Vert _{X_{T}}=\underset{0<t\leq T}{\sup}t^{{1}/{4}}\Vert  f(t)\Vert _{L^{4/3}(\Omega)},\;\Vert h\Vert _{Z_{T}}=\underset{0<t\leq T}{\sup}t^{{3}/{8}}\Vert  h(t)\Vert _{L^{4/3}(\RR^3)}.
\end{equation*}
The local wellposedness of \eqref{int-equation} is then given by the following proposition


\begin{prop}\label{existence Locale-mu} There exist non negative constants $\EE $ and $C$ such that the following hold. Let  $(\omega_0,\rho_{0})\in\mathscr{M}(\Omega)\times L^1(\RR^3)$ with $\rho_0$ axisymmetric and  $\|\omega_{0,pp}\|_{\mathscr{M}(\Omega)}\le\EE$, then, \eqref{int-equation} admits a unique local solution $(\omega_\theta,\widetilde{\rho},\rho)$, defined for all positive  $t\leqslant T =T(\omega_0,\rho_0) $ such that 
\begin{equation}\label{Espace-sol}
(\omega_{\theta},\widetilde{\rho},\rho)\in C\big((0,T];X_T\big)\times C\big((0,T];X_T\big)\times C\big((0,T];Z_T\big). 
\end{equation}

Moreover, if the size of initial data is small enough, the local time of existence $T$ can be taken arbitrarily large. 
\end{prop}
\begin{proof} We closely follow the demonstration of Proposition 3.1 in \cite{Hanachi-Houamed-Zerguine} with minor modifications due to the particularity of initial data. In view of Proposition \ref{P,S1,S2}, Proposition \ref{Prop-W-E-1}  and Proposition \ref{prop mu-lamda}, we have for $T>0$
\begin{equation}\label{initial1}
\sup_{0<t\leq T}t^{{1}/{4}}\|\Ss_1(t)\omega_0\|_{L^{\frac43}(\Omega)}\leq C\|\omega_0\|_{\mathscr{M}(\Omega)}.
\end{equation}
and
\begin{equation}\label{initial2}
\sup_{0<t\leq T}t^{{1}/{4}}\|\Ss_1(t)\widetilde{\rho} _0\|_{L^{\frac43}(\Omega)}\leq C\|r\rho_0\|_{L^1(\Omega)} = C\| \rho_0\|_{L^1(\RR^3)}.
\end{equation}
On the other hand, the fact that
\begin{equation*}
\|\Ss_2(t)\rho_0\|_{L^{\frac43}(\RR^3)}= \|r^\frac{3}{4}\Ss_2(t)\rho_0\|_{L^{\frac43}(\Omega)}
\end{equation*}
together with the first estimate stated in Proposition \ref{Prop-W-E-1}, we further get
\begin{equation}\label{initial3}
\sup_{0<t\leq T}t^{{3}/{8}}\|\Ss_2(t)\rho_0\|_{L^{\frac43}(\RR^3)} \leq  C\|r\rho_0\|_{L^1(\Omega)}=C\|\rho_0\|_{L^1 (\RR^3)}.
\end{equation}
Combining \eqref{initial1}, \eqref{initial2} and \eqref{initial3} to obtain  $\big(\omega_{\lin},\widetilde{\rho}_{\lin},\rho_{\lin}\big)\in\mathscr{X}_T$ with $$(\omega_{\lin}(t),\widetilde{\rho}_{\lin}(t),\rho_{\lin}(t))=(\Ss_1(t)\omega_0,\Ss_1(t)\widetilde{\rho}_0,\Ss_2(t)\rho_0) $$
and 
$$\mathscr{X}_T \triangleq X_T\times X_T\times Z_T $$
 
Now, from \eqref{initial1}, \eqref{initial2} and \eqref{initial3}, we have
\begin{equation}\label{T-qlq}
\Lambda(\omega_0,\widetilde{\rho}_0,\rho_0,T)\triangleq C\|(\omega_\lin,\widetilde{\rho}_\lin,\rho_\lin)\|_{\mathscr{X}_T}\leq C_0\big( \|\omega_0\|_{\mathscr{M}(\Omega)}+\|\rho_0\|_{L^1(\RR^3)}\big) .
\end{equation}
Moreover,  according to \eqref{Est2-S(t)} and \eqref{Eq-3D-heat kernel-measure}, we have\footnote{Remark that since $\rho_0\in L^1(\RR^3)$, then  $\rho_{0,pp}=0.$}
\begin{equation}\label{T-petit}
\limsup_{T\rightarrow 0}\Lambda(\omega_0,\widetilde{\rho}_0,\rho_0,T)\le\EE C
\end{equation}
The estimates of the bilinear terms can be done as in the proof of Proposition 3.1 in \cite{Hanachi-Houamed-Zerguine} . Hence, we get the nonlinear system
 \begin{equation}\label{omega-XT}
 \left\{\begin{array}{ll}
\| \omega_\theta\|_{X_T} \leq  \|\omega_{\lin}\|_{X_T} + C  \| \omega_\theta\|_{X_T}^2 +  C \| \rho\|_{Z_T} \\~~\\
\| \widetilde{\rho}\|_{X_T} \leq   \|\widetilde{\rho}_{\lin}\|_{X_T}  +  C \| \omega_\theta\|_{X_T}\| \widetilde{\rho}\|_{X_T}  +  C \| \rho\|_{Z_T} \\~~\\
\| \rho\|_{Z_T} \leq  \| \rho_{\lin}\|_{Z_T} +  C \| \omega_\theta\|_{X_T}\| \widetilde{\rho}\|_{X_T}.
\end{array}\right.
\end{equation}

for some universal constant $C>0$. By substituting $\|\rho\|_{Z_T}$ in the two first equations of \eqref{omega-XT}, we readily get
\begin{equation}\label{contraction condition}
\| \omega_\theta\|_{X_T}+\| \widetilde{\rho}\|_{X_T} \leq   \Lambda(\omega_{0},\widetilde{\rho}_0, \rho_0,T)+\widetilde{C}\big(\| \omega_\theta\|_{X_T}+ \|\widetilde{\rho}\|_{X_T}\big)^2.
\end{equation}
To complete the contraction argument, let us fix $R>0$ such that $2\widetilde{C}R<1$ and define the ball $$\mathcal{B}_T(R)\triangleq \{(a, b) \in   X_T\times X_T : \, \|( a,b)\|_{X_T\times X_T} <R \},$$ for $(\omega_\theta, \widetilde{\rho})\in \mathcal{B}_T(R)$ the contraction argument is satisfied if $\Lambda(\omega_0,\widetilde{\rho}_0,\rho_0,T)\leq R/2$.  The last requirement can be realized in either case
\begin{enumerate}
\item[{\bf---}]$C_0\big(\|\omega_0\|_{\mathscr{M}(\Omega)}+\|\rho_0\|_{L^1(\RR^3)}\big)\leq R/2$ for any $T>0$, or
\item[{\bf---}]$C\EE\leq R/2$ for $T>0$ is small enough, depending on $\omega_{0,pp}$ and $\rho_0$ (this is possible because $\Lambda(\omega_0,\widetilde{\rho}_0,\rho_0,T)\rightarrow C\EE$ when $T\rightarrow 0$).
\end{enumerate} 
 In other words, we can prove the global well-posedness if the initial data is sufficiently small, or the local well-posedness if only the atomic part $\omega_{0,pp}$ is small. The rigorous construction of the solution can be done by the standard fixed point  schema (see \cite{Hanachi-Houamed-Zerguine}). This concludes the proof of Proposition \ref{existence Locale-mu}.
 \end{proof}
 Remark that the local solution constructed above becomes instantly integrable after time $t>0$. Hence, all the a priori estimates proved in our previous work \cite{Hanachi-Houamed-Zerguine} remains valid for all $t>0$.  However, for the sake of completeness, we provide in the following proposition the precise statement of more properties of our solution. 
 \begin{prop}\label{properties-decay-Lp} Let $(\omega_\theta,\widetilde{\rho},\rho)$ be the solution of \eqref{int-equation} obtained by Proposition \ref{existence Locale-mu} associated to initial data $(\omega_0,\widetilde{\rho}_0,\rho_0)\in \mathscr{M}(\Omega)\times L^1(\Omega)\times L^1(\RR^3)$. Then for any $p\in (1,\infty]$, we have
 \begin{equation}\label{omega-lp}
 \underset{t\rightarrow 0}{\lim}\, t^{(1-\frac{1}{p})} \| \omega_\theta(t) \|_{L^p(\Omega)} \le C\|\omega_{0,pp}\|_{\mathscr{M}(\Omega)},
 \end{equation}
  \begin{equation}\label{rrho-lp}
 \underset{t\rightarrow 0}{\lim}\, t^{(1-\frac{1}{p})} \| \widetilde{\rho}(t) \|_{L^p(\Omega)} = 0,
 \end{equation}
  \begin{equation}\label{rho-lp}
 \underset{t\rightarrow 0}{\lim}\, t^{\frac{3}{2} (1-\frac{1}{p})} \| \rho(t) \|_{L^p(\RR^3)} = 0.
 \end{equation}
 For $p=1$, the above quantities are bounded as $t \rightarrow 0$.  
 \end{prop}
 \begin{Rema}
 As aforementioned, Proposition \ref{properties-decay-Lp} can be proved along the same lines as   the proof of  Proposition 11 in \cite{Hanachi-Houamed-Zerguine}. However, we should mention that the case $p=\infty$ is missing \cite{Hanachi-Houamed-Zerguine}, therefore, we provide below a complementary proof that treats this case as well.
 \end{Rema}
 \begin{proof}
Let us first recall the following notation from \cite{Hanachi-Houamed-Zerguine}
 
  \begin{equation*}
 N_p(f,T) \triangleq \underset{0<t\leq T}{\sup} t^{(1-\frac{1}{p})} \|f \|_{L^p(\Omega)},\quad J_p(f,T) \triangleq \underset{0<t\leq T}{\sup}\;t^{\frac{3}{2} (1-\frac{1}{p})} \|  f \|_{L^p(\RR^3)}.
 \end{equation*}
\begin{equation*}
  M_p (f_0, T) \triangleq  \underset{0<t\leq T}{\sup}\;t^{(1-\frac{1}{p})} \|  \Ss_1(t) f_0  \|_{L^p(\Omega)}, \quad F_p  (f_0, T) \triangleq  \underset{0<t\leq T}{\sup}\;t^{\frac{3}{2} (1-\frac{1}{p})} \|  \Ss_2(t)f_0  \|_{L^p(\RR^3)}.
 \end{equation*}
 According to Proposition \ref{P,S1,S2}, Proposition \ref{Prop-W-E-1}  and Proposition \ref{prop mu-lamda} we find for all $p\in (1,\infty]$
 \begin{equation}\label{data-limit1}
\lim_{t\rightarrow 0} M_p (\omega_0, T) \le C\|\omega_{0,pp}\|_{\mathscr{M}(\Omega)}
 \end{equation}
 and
 \begin{equation}\label{data-limit2}
\lim_{t\rightarrow 0} M_p (\widetilde{\rho}_0, T)=\lim_{t\rightarrow 0}F_p  (\rho_0, T) =0.
 \end{equation}
 
 Thanks to the Proposition \ref{existence Locale-mu}, the result in the case $ p = \frac43 $ is already proved.  By interpolation we find the result for $p\in(1, \frac43 ]$ as long as the $ L^1(\Omega)\times L^1(\RR^3 )-$ norm of $ (\omega_ {\theta}(t), \rho (t))$ remains bounded in a neighborhood of $t=0 $. Let us suppose for a moment that this is true and we get back to prove this claim later.\\
 
Doing so, it remains to prove the result for $p> \frac43$. For this purpose, we can proceed by a bootstrap argument as in the proof of Proposition 11 \cite{Hanachi-Houamed-Zerguine}.

In view of Proposition 3 in \cite{Hanachi-Houamed-Zerguine}, Proposition \ref{P,S1,S2} and Proposition \ref{Prop-W-E-1}, we have
\begin{align*}
\| \omega_\theta(t) \|_{L^p(\Omega)}  \leq \|\Ss_1(t) \omega_0 \|_{L^p(\Omega)}  &+ C \int_0^\frac{t}{2} \frac{\|\omega_\theta \|^2_{L^q(\Omega)}}{ (t-\tau) ^{\frac{2}{q} - \frac{1}{p} }}d\tau + C \int_{ \frac{t}{2}}^t   \frac{\|\omega_\theta(\tau) \|_{L^{q_1}(\Omega)}\|\omega_\theta(\tau) \|_{L^{q_2}(\Omega)}}{ (t-\tau) ^{\frac{1}{q_1} +\frac{1}{q_2} - \frac{1}{p} }}d\tau\\
&+ C \int_0 ^\frac{t}{2} \frac{\| \rho(\tau) \|_{L^\frac{4}{3} (\RR^3)}}{(t-\tau)^{  \frac{13}{8}- \frac{1}{p}} } d\tau + C \int_\frac{t}{2}^t  \frac{\| \rho(\tau) \|_{L^p (\RR^3)}}{(t-\tau)^{ \frac{1}{2} + \frac{1}{2p} }}d\tau.
\end{align*} 

Under the conditions
\begin{equation}\label{bootstrap condition 1}
\frac{1}{2} \leq  \frac{2}{q} - \frac{1}{p} , \quad \frac{1}{2} \leq \frac{1}{q_1} +\frac{1}{q_2} - \frac{1}{p} <1,
\end{equation}
 Thus, straightforward computation yields
 \begin{equation}\label{N_p-omega-inequa}
 N_p(\omega_\theta,T) \leq M_p (\omega_0,T) + C_{p,q}  N_q(\omega_\theta,T)^2+C_{q_1,q_2}  N_{q_1}(\omega_\theta,T) N_{q_2}(\omega_\theta,T) + C_p J_{\frac{4}{3}}(\rho,T)+ C_p J_p(\rho,T).
 \end{equation}
Since $\widetilde{\rho}$ satisfies a similar equation to that of $\omega_{\theta}$, we infer that
\begin{eqnarray} \label{N_p-tilde rho inequa}
 N_p(\widetilde{\rho},T) \leq M_p (\widetilde{\rho}_0,T) &+& C_{p,q}  N_q(\omega_\theta,T)  N_q(\widetilde{\rho} ,T)\nonumber \\
 &+& C_{q_1,q_2}  N_{q_1}(\omega_\theta,T) N_{q_2}(\widetilde{\rho},T) + C_p J_{\frac{4}{3}}(\rho,T)+ C_p J_p(\rho,T).
\end{eqnarray} 
Finally, similar arguments yield\footnote{We refer also to inequality 77 from \cite{Hanachi-Houamed-Zerguine} for more details.} 
 \begin{equation}\label{J_p-inequa}
J_p(\rho,T) \leq F_p (\rho_0,T) + C_{p}  N_\frac{4}{3} (\omega_\theta,T)  N_\frac{4}{3} (\widetilde{\rho} ,T)+C_{q_1,q_2}  N_{q_1}(\omega_\theta,T) N_{q_2}(\widetilde{\rho},T),  
\end{equation}
for any $q_1,q_2 $ such that 
\begin{equation}\label{cond:2}
\frac{1}{q_1}+\frac{1}{q_2}-\frac{3}{2p}<\frac12.
\end{equation}  
Now, by plugging \eqref{J_p-inequa} in \eqref{N_p-omega-inequa} and \eqref{N_p-tilde rho inequa}, we find  for $q=\frac{4}{3}$
\begin{eqnarray*} 
\nonumber N_p(\omega_\theta,T) &\leq& C_{p,q_1,q_2}\big( M_p (\omega_0,T) +F_p (\rho_0,T) +  N_{\frac{4}{3}}(\omega_\theta,T)^2+ N_\frac{4}{3} (\omega_\theta,T)  N_\frac{4}{3} (\widetilde{\rho},T)\\
 &&+J_{\frac{4}{3}}(\rho,T)+N_{q_1}(\omega_\theta,T) N_{q_2}(\omega_\theta,T)+  N_{q_1}(\omega_\theta,T) N_{q_2}(\widetilde{\rho},T) \big),
\end{eqnarray*}
and
\begin{eqnarray*} 
N_p(\widetilde{\rho},T)  \leq C_{p,q_1,q_2}\big( M_p (\widetilde{\rho}_0,T) &+&  F_p (\rho_0,T)+  N_{\frac43}(\omega_\theta,T)  N_{\frac43}(\widetilde{\rho} ,T)\nonumber\\
&+&  J_{\frac{4}{3}}(\rho,T)+ N_{q_1}(\omega_\theta,T) N_{q_2}(\widetilde{\rho},T) \big).
\end{eqnarray*}

 Now, to cover all the range $p\in (\frac{4}{3}, \infty)$, we can proceed by a bootstrap argument as in \cite{Hanachi-Houamed-Zerguine}. The only difference we should point out here is the fact that 
 $\displaystyle\lim_{T\rightarrow 0} M_p (\omega_0,T) $ is not necessary zero, but, it satisfies 
 \begin{equation*}
 \lim_{T\rightarrow 0} M_p (\omega_0,T)  \leqslant C\|\omega_{0,pp}\|_{\mathscr{M}(\Omega)}, \quad \forall p\in (1,\infty].
 \end{equation*}
 Thus, we obtain
\begin{equation*}
\lim_{T\rightarrow 0} N_p(\omega_\theta,T) \leqslant  C\|\omega_{0,pp}\|_{\mathscr{M}(\Omega)},\quad\textnormal{and}\quad \lim_{T\rightarrow 0}N_p(\widetilde{\rho},T)= 0, \quad \text{for all } p\in (1,\infty).
\end{equation*}
Finally, substituting this latest into \eqref{J_p-inequa}, leads to
\begin{equation*}
\lim_{T\rightarrow 0}J_p(\rho,T) = 0, \quad \text{for all } p\in (1,\infty) .
\end{equation*}
In order to treat the case $p=\infty$, we need to avoid some technical issues arising from the restriction \eqref{cond:2}. To this end, we   chose first $q_1=\frac{3}{2}$, $q_2=4$ and $p=\infty$ in \eqref{N_p-omega-inequa} and \eqref{N_p-tilde rho inequa}. Remark that this choice of $(q_1,q_2,p)$ is admissible by the relation \eqref{bootstrap condition 1}, hence, we obtain
\begin{equation}\label{N1}
N_\infty(\omega_\theta,T) \leq    f_1(T) + C J_\infty(\rho,T).
\end{equation}
\begin{equation}\label{N2}
N_\infty(\widetilde{\rho},T) \leq   f_2(T) + C J_\infty(\rho,T),
\end{equation}
with 
$$f_1(T) = M_\infty(\omega_0,T) +CN_\frac43(\omega_\theta,T)^2 + CN_\frac32(\omega_\theta,T)N_4(\omega_\theta,T) +C J_\frac43(\rho,T) \underset{T\rightarrow0}{\leqslant }C\|\omega_{0,pp}\|_{\mathscr{M}(\Omega)}, $$
$$f_2(T) =  M_\infty(\widetilde{\rho}_0,T) +CN_\frac43(\omega_\theta,T)N_\frac43(\widetilde{\rho},T) + CN_\frac32(\omega_\theta,T)N_4(\widetilde{\rho},T) +C J_\frac{4}{3}(\rho,T) \underset{T\rightarrow0}{\longrightarrow 0} . $$
Now, we need to deal with $J_\infty(\rho,T)$. by exploring again the properties of the heat semi-group as in the case $p<\infty$, we infer that
\begin{equation}\label{rho-infiniti-1}
 \norm{\rho(t)}_{L^\infty(\mathbb{R}^3)}\leq \norm {\mathbb{S}_2(t)\rho_0}_{L^\infty(\mathbb{R}^3)}  + C \int_0^\frac{t}{2} \frac{\|\omega_\theta \| _{L^\frac{4}{3} (\Omega)}\|\widetilde{\rho} \| _{L^\frac{4}{3} (\Omega)} }{ (t-\tau) ^{ \frac{1}{2} + \frac{3}{2} }}d\tau + C \int_{ \frac{t}{2}}^t  \frac{\norm{v(\tau)}_{L^\infty(\mathbb{R}^3)}\norm{\rho(\tau) }_{L^q(\mathbb{R}^3)}}{(t-\tau)^{\frac{1}{2}+\frac{3}{2q}}}  d\tau.
 \end{equation}
To assert that the last term on the r.h.s above is finite, we need to chose $q$ such that $\frac{1}{2} + \frac{3}{2q}<1 $. A possible choice is then $q=6$.  On the other hand, remark that due to the Biot-Savart law, we have, for some $1<m<2<\ell<\infty$
\begin{equation*}
\norm {v(\tau)}_{L^\infty(\Omega)}\lesssim \norm{\omega_\theta(\tau)}_{L^m(\Omega)}^\alpha\norm{\omega_\theta(\tau)}_{L^\ell(\Omega)}^{1-\alpha}, \quad \text{ for } \alpha= \frac{m}{2}\frac{\ell-2}{\ell-m} \in (0,1).
\end{equation*}
Moreover, since $m,\ell<\infty$, the previous estimates of $\omega_\theta$ together with a  straightforward computation yield  
\begin{equation}\label{u-infinity estimate}
\norm {v(\tau)}_{L^\infty(\Omega)} \lesssim \tau^{-\frac{1}{2}}.
\end{equation}
Finally, we infer that
\begin{equation*}
t^\frac{3}{2}\norm{\rho(t)}_{L^\infty(\mathbb{R}^3)} \leq    t^\frac{3}{2}\norm {\mathbb{S}_2(t)\rho_0}_{L^\infty(\mathbb{R}^3)}  + C N_\frac{4}{3}(\omega_\theta,T)N_\frac{4}{3}(\widetilde{\rho},T) + J_6(\rho,T)\sup_{\tau\in(0,T)}\Big(\tau ^\frac{1}{2}\norm{v(\tau)}_{L^\infty(\Omega)}\Big).
\end{equation*}
It is easy then to conclude that 
 \begin{equation*}
\lim_{T\rightarrow 0}J_\infty(\rho,T) = 0, 
\end{equation*}
and, therefore, we get from \eqref{N1} and \eqref{N2}
 \begin{equation*}
\lim_{T\rightarrow 0}N_\infty(\omega_\theta,T)\leqslant C\|\omega_{0,pp}\|_{\mathscr{M}(\Omega)},\quad\textnormal{and}\quad \lim_{T\rightarrow 0}N_\infty(\widetilde{\rho},T)=0.
\end{equation*}
This ends the proof of Proposition \ref{properties-decay-Lp} provided that we prove the following claim
\begin{equation*}
\norm{(\omega_\theta(t),\widetilde{\rho}(t),\rho(t))}_{L^1(\Omega)\times L^1(\Omega) \times L^1(\RR^3)} \lesssim \norm{(\omega_0,\rho_0)}_{\mathscr{M}(\Omega) \times L^1(\RR^3)}.
\end{equation*}
From the definition of $\widetilde{\Gamma}$ and the fact $\widetilde{\rho}= r\rho$, the above claim is equivalent to

\begin{equation}\label{estimations-L1-local}
\big\|(\widetilde{\Gamma}(t) ,\rho(t))\|_{L^1(\Omega)  \times L^1(\RR^3)} \lesssim \norm{(\omega_0,\rho_0)}_{\mathscr{M}(\Omega)\times L^1(\RR^3)}.
\end{equation}
Let us then prove that \eqref{estimations-L1-local}. We will restrict ourselves to the estimates of the nonlinear terms since the linear parts can be treated by applying the properties of semi-groups recalled in the previous section. So, according to the equations of $\widetilde{\Gamma}$ and $\rho$, we must show that 

\begin{equation}\label{claim1-local existence}
\int_0^t \norm {\Ss_1 (t-\tau) \Div_{\star}(v \widetilde{\Gamma})(\tau)}_{L^1(\Omega)} d\tau \lesssim \norm{(\omega_0,\rho_0)}_{L^1(\Omega) \times L^1(\RR^3)}
\end{equation}
and 
\begin{equation}\label{claim2-local existence}
\int_0^t \norm {\Ss_2 (t-\tau) \Div(v \rho)(\tau)}_{L^1(\RR^3)} d\tau \lesssim \norm{(\omega_0,\rho_0)}_{L^1(\Omega) \times L^1(\RR^3)}.
\end{equation}
For \eqref{claim1-local existence}, H\"older's inequality, Biot Savart law and the definition of the space $X_T$ lead to
\begin{align*}
\int_0^t \norm {\Ss_1 (t-\tau)\Div_{\star}(v \widetilde{\Gamma})(\tau)}_{L^1(\Omega)} d\tau &\lesssim \int_0^t \frac{1}{(t-\tau)^\frac{1}{2}}\norm {v(\tau)}_{L^4(\Omega)}  \norm {\widetilde{\Gamma}(\tau)}_{L^\frac{4}{3}(\Omega)} d\tau\\
&\lesssim \int_0^t \frac{1}{(t-\tau)^\frac{1}{2}}\norm {\omega_\theta(\tau)}_{L^\frac{4}{3}(\Omega)}  \big\|\widetilde{\Gamma}(\tau)\big\|_{L^\frac{4}{3}(\Omega)} d\tau\\
&\lesssim \int_0^t \frac{1}{(t-\tau)^\frac{1}{2} \tau^{\frac{1}{2}}} d\tau\norm {\omega_\theta }_{X_T}  \big\|\widetilde{\Gamma}\|_{X_T}\\
&\lesssim \norm {\omega_\theta }_{X_T}  \big\|\widetilde{\Gamma}\|_{X_T}.
\end{align*}
To treat  \eqref{claim2-local existence}, we use the fact that
$$\norm {\Ss_2 (t-\tau)\Div(v \rho)(\tau)}_{L^1(\RR^3)} = \norm {r\Ss_2 (t-\tau)\Div(v \rho)(\tau)}_{L^1(\Omega)},$$
then we use first Proposition \ref{Prop-W-E-1} to infer that
\begin{align*}
\int_0^t \norm {r\Ss_2 (t-\tau)\Div(v \rho)(\tau)}_{L^1(\Omega)} d\tau \lesssim \int_0^t \frac{1}{(t-\tau)^\frac{1}{2}}\norm {  v r\rho (\tau)}_{L^1(\Omega)} d\tau.
\end{align*}
Therefore, the identity $\widetilde{\rho}= r \rho$,  H\"older's inequality and the Biot--Savart law yield
\begin{eqnarray*}
\norm {\Ss_2 (t-\tau)\Div(v \rho)(\tau)}_{L^1(\RR^3)} &\lesssim & \int_0^t \frac{1}{(t-\tau)^\frac{1}{2}}\norm {v(\tau)}_{L^4(\Omega)}  \norm {\widetilde{\rho}(\tau)}_{L^\frac{4}{3}(\Omega)} d\tau\\
&\lesssim &\|\omega_\theta\|_{X_T}\|\widetilde{\rho}\|_{X_T}.
\end{eqnarray*}
Hence, the estimates in $X_T $ lead to \eqref{claim1-local existence} and \eqref{claim2-local existence}, thereafter, \eqref{estimations-L1-local} follows.
 \end{proof}
 
\hspace{0.5cm} To complete the proof of the Theorem \ref{Second-Th.}, it remains to outline the proof of the continuity of the solution and the convergence to the initial data. Precisely, we shall prove the following
\begin{prop}\label{Continuity of the solution}
Let $(\omega_0,\rho_0)$ be the initial data to system \eqref{VD} that satisfies the assumptions of  Theorem \ref{Second-Th.}. Let $(\omega_\theta,\rho)$ be the local solution given by the fixed point argument such that
\begin{equation*} 
(\omega_{\theta}, r\rho) \in \Big(  L^\infty\big((0,T);  L^1(\Omega) \cap  L^\infty(\Omega)\big)  \Big) \times \Big(L^\infty\big([0,T);  L^1(\Omega)\big)\cap L^\infty\big((0,T); L^\infty(\Omega)\big) \Big),
\end{equation*}
\begin{equation*} 
 \rho \in L^\infty\big([0,T);  L^1(\RR^3)\big)\cap L^\infty\big((0,T); L^\infty(\RR^3)\big).
\end{equation*}
Then
\begin{equation}\label{result1}
(\omega_{\theta}, r\rho) \in \Big(  C^0\big((0,T);  L^1(\Omega) \cap  L^\infty(\Omega)\big)  \Big) \times \Big(  C^0\big([0,T);  L^1(\Omega)\big)\cap C^0\big((0,T); L^\infty(\Omega)\big) \Big), 
\end{equation}

\begin{equation}\label{result2}
 \rho \in C^0\big([0,T);  L^1(\RR^3)\big)\cap C^0\big((0,T); L^\infty(\RR^3)\big).
\end{equation}
Moreover, we have the following convergence to the initial  
\begin{equation}\label{convergence-initial-vorticity}
\omega_\theta(t) \rightharpoonup \omega_0, \quad \text{ as } t \rightarrow 0 
\end{equation}
and 
\begin{equation}\label{convergence-initial-density}
\lim_{t\rightarrow 0 }\|\rho(t)-\rho_0 \|_{L^1(\RR^3)} =0.
\end{equation}
\end{prop}
\begin{proof}
Assertions \eqref{result1} and \eqref{result2} concern the continuity of the solution away from $0$, this can be done by the same way as in our previous work \cite{Hanachi-Houamed-Zerguine}  since the solution satisfies, for all $t_0\in (0,T]$  
  $$\big(\omega(t_0),\widetilde{\rho}(t_0),\rho(t_0)\big)\in L^p(\Omega)\times L^p(\Omega) \times L^p(\RR^3), \quad \forall p\in [1,\infty].$$ 

Let us now investigate the convergence to the initial data \eqref{convergence-initial-vorticity} and \eqref{convergence-initial-density}. It should be noted that the major difficulty in this part is the weak convergence of the vorticity towards the initial datum.  We begin with the proof of the limit \eqref{convergence-initial-density} which does not differ a lot from the proof  given in \cite{Hanachi-Houamed-Zerguine}.

Indeed, \eqref{convergence-initial-density} is an easy consequence of 
\begin{equation}\label{claim6}
\limsup_{t\rightarrow 0}\| \rho(t)- \mathbb{S}_2(t)\rho_0 \|_{L^1(\RR^3)}=0
\end{equation}
Hence, we should focus on the proof of \eqref{claim6}, for $t>0$. By using Proposition \ref{Prop-W-E-1} for $\alpha= \beta = 1$, we get
\begin{align*}
\| \rho(t)- \mathbb{S}_2(t)\rho_0 \|_{L^1(\RR^3)} &\leq  \int_0^t \| r\mathbb{S}_2(t) \Div(v(\tau) \rho(\tau) ) \|_{L^1(\Omega)}d\tau \\
&\lesssim \int_0^t \frac{1}{(t-\tau)^\frac{1}{2}} \| v(\tau) r\rho(\tau) ) \|_{L^1(\Omega)}d\tau \\
&\lesssim \int_0^t \frac{1}{(t-\tau)^\frac{1}{2}} \| v(\tau) \|_{L^4(\Omega)} \| r\rho(\tau) ) \|_{L^\frac{4}{3} (\Omega)}d\tau.
\end{align*}
Then, Biot--Savart's law yields,
\begin{align*}
\| \rho(t)- \mathbb{S}_2(t)\rho_0 \|_{L^1(\RR^3)}&\lesssim \int_0^t \frac{1}{(t-\tau)^\frac{1}{2}} \| \omega_\theta(\tau) \|_{L^\frac{4}{3} (\Omega)} \| \widetilde{\rho}(\tau) ) \|_{L^\frac{4}{3} (\Omega)}d\tau \\
&\lesssim  \|\omega_\theta \|_{X_T} \|\widetilde{\rho}\|_{X_t}\int_0^t \frac{1}{(t-\tau)^\frac{1}{2} \tau^\frac{1}{2}}\\
&\leq C \|\omega_\theta \|_{X_T} \|\widetilde{\rho}\|_{X_t}.
\end{align*}
Thus \eqref{claim6} follows from the fact that $\displaystyle\lim_{t\rightarrow 0} \|\widetilde{\rho}\|_{X_t} = 0$.\\
We turn   now to prove \eqref{convergence-initial-vorticity} and we follow the idea in \cite{Gallay-Sverak}. We begin by proving the following claim
\begin{equation}\label{claim7}
\limsup_{t\rightarrow 0}\| \omega_\theta(t)- \mathbb{S}_1(t)\omega_0 \|_{L^1(\Omega)}=0
\end{equation}
As mentioned earlier, the linear term of $\partial_r \rho$  is a hurdle in the $L^1$-estimate of $\omega_\theta$. To avoid the estimation of this term, we use the coupling $\widetilde{\Gamma} = \omega_\theta - \frac{\widetilde{\rho}}{2}$. First, note that \eqref{convergence-initial-density} gives

 \begin{equation}\label{L1:rrho}
\limsup_{t\rightarrow 0}\| \widetilde{\rho} (t)- \mathbb{S}_1(t)\widetilde{\rho}_0 \|_{L^1(\Omega)}=0.
\end{equation}
Thus \eqref{claim7} is equivalent to the following assertion
\begin{equation}\label{claim8}
\limsup_{t\rightarrow 0}\| \widetilde{\Gamma}(t)- \mathbb{S}_1(t)\widetilde{\Gamma}_0 \|_{L^1(\Omega)}=0,
\end{equation}
where, $\widetilde{\Gamma}_0 = \omega_0- \frac{\widetilde{\rho}_0}{2} \in \mathscr{M}(\Omega)$. Let us define the functional $\mathcal{F}$ by
\begin{equation*}
(\mathcal{F}g)(t) \triangleq \int_0^t \mathbb{S}_1(t-\tau)\Div_{\star} (v(\tau) g(\tau) ) d\tau.
\end{equation*}
We emphasize that the following estimate holds true, for any $g\in X_T$
\begin{equation}\label{ES:FG}
\| \mathcal{F}g(t)\| _{L^1(\Omega)} + \| \mathcal{F}g\| _{X_t}\leqslant \widetilde{C}\|\omega_\theta \|_{X_t}\|g \|_{X_t}, \quad \forall t\leqslant T.
\end{equation}
The proof of \eqref{ES:FG} can be done by using the estimates of subsection \ref{Semi-group estimates}.
On the other hand, from the equation of $\widetilde{\Gamma}$, we have
\begin{equation}\label{continuity-difference-gamma}
 \widetilde{\Gamma} - \widetilde{\Gamma}_{\lin}  = \big(\mathcal{F}(\widetilde{\Gamma}_{\lin})-\mathcal{F}(\widetilde{\Gamma}) \big) - \mathcal{F}(\widetilde{\Gamma}_{\lin})
\end{equation}
where $\widetilde{\Gamma}_{\lin}(\cdot) = \mathbb{S}_1(\cdot)\widetilde{\Gamma}_0$.
Let $R$ be the radius of the ball in which we applied the fixed-point argument to construct the local solution\footnote{That is to say, $R$ is such that $2\widetilde{C}R<1$,  where $\widetilde{C}>0$ is defined by \eqref{contraction condition}}. Hence, we have
\begin{equation*}
\| \widetilde{\rho}_{\lin} \|_{X_T} + \| \omega_{\lin} \|_{X_T}+ \| \widetilde{\rho} \|_{X_T} + \| \omega_\theta  \|_{X_T} \leq 2 R
\end{equation*}
Let us also define the two quantities $\delta$ and $ \ell_p(\omega_0)$ by 
\begin{equation*}
\delta \triangleq \limsup_{T\rightarrow 0} \| \widetilde{\Gamma}- \widetilde{\Gamma}_{\lin} \|_{ X_T}
\end{equation*}
and 
\begin{equation*}
\ell_p(\omega_0) \triangleq\limsup_{t\rightarrow 0} t^{1-\frac{1}{p}} \|\mathcal{F}(\omega_{\lin})(t) \|_{L^p(\Omega) }, \quad p\in[1,\infty].
\end{equation*}
Note first that, we have $$\limsup_{T\rightarrow 0}\| \widetilde{\rho}- \widetilde{\rho}_{\lin}\|_{ X_T} =\limsup_{T\rightarrow 0}\|\mathcal{F}(\rho_{\lin}) \|_{X_T}=0.$$
Similarly, by a bootstrap argument (Proposition \ref{properties-decay-Lp}), we can prove that
$$ \limsup_{t\rightarrow 0} t^{1-\frac{1}{p}} \|\mathcal{F}(\rho_{\lin}) \|_{L^p(\Omega) }=0, \quad \forall p\in[1,\infty]. $$ 
Thus, by definition of $\widetilde{\Gamma}_{\lin}$ and by linearity of the functional $\mathcal{F}$, we deduce that
\begin{equation}\label{GAMMA-OMEGA}
\limsup_{t\rightarrow 0} t^{1-\frac{1}{p}} \|\mathcal{F}(\widetilde{\Gamma} _{\lin}) \|_{L^p(\Omega) }= \ell_p(\omega_0) , \quad p\in[1,\infty].
\end{equation}
We resume now from \eqref{continuity-difference-gamma}. Note that \eqref{ES:FG} together with \eqref{GAMMA-OMEGA} yield
\begin{align*}
\delta&\leq  \widetilde{C} \limsup_{t\rightarrow 0}\bigg( \| \omega_\theta \|_{X_t} \| \widetilde{\Gamma}-\widetilde{\Gamma}_{\lin} \|_{X_t}\bigg)   + \ell_\frac{4}{3}(\omega_0)\\
&\leq 2\widetilde{C} R \delta+ \ell_\frac{4}{3}(\omega_0).
\end{align*}

We end up with $\delta=0$ because, $\ell_p(\omega_0)=0$, for all $p\in[1,\infty]$ \footnote{See the last part in \cite[Section 4]{Gallay-Sverak} and   \cite[Section 2.3.4]{Gallay} for a detailed proof of the fact $\ell_p(\omega_0)=0$.}, and $2\widetilde{C}R<1$.\\

We are now in position to prove \eqref{claim8}. Again, using \eqref{continuity-difference-gamma} and \eqref{ES:FG} infer that
\begin{align*}
\limsup_{t\rightarrow 0} \| \widetilde{\Gamma}(t)-\widetilde{\Gamma}_{\lin}(t) \|_{L^1(\Omega)}&\leq C \| \omega_\theta  \|_{X_T}\limsup_{t\rightarrow 0}\| \widetilde{\Gamma}-\widetilde{\Gamma}_{\lin} \|_{X_t}+ \ell_1(\omega_0) = 0,
\end{align*}
where, we have used the fact that  $\delta=\ell_1(\omega_0) = 0$.
Consequently, we obtain, in view of \eqref{L1:rrho}
$$
\limsup_{t\rightarrow 0} \| \omega_\theta(t)-\omega _{\lin}(t) \|_{L^1(\Omega)}=0.
$$
This ends the hard part of the proof. To prove the weak limit towards to initial vorticity, we only need to use the fact that\footnote{See  \cite[Section 2]{Hanachi-Houamed-Zerguine}.}
$$
\omega_{\lin}(t,r,z)=\frac{1}{4\pi t}\int_{\Omega}\frac{\tilde{r}^{1/2}}{r^{1/2}}\mathscr{N}_1\Big(\frac{t}{r\tilde{r}}\Big)e^{-\frac{(r-\tilde{r})^2+(z-\tilde{z})^2}{4t}}d\omega_0(\tilde{r},\tilde{z}).
$$
From the above formula, we can check then that $\omega_{\lin} \rightharpoonup \omega_0$ and finally \eqref{convergence-initial-vorticity} follows. The proof of the proposition is then achieved.
\end{proof}
\begin{Rema}
We should point out that, the propositions we proved in this subsection do not say anything about the global well-posedness, it is all about the local theory. However, one can prove that the local solution we construct in Proposition \ref{existence Locale-mu} can be in fact extended to be global in time. We postpone the details of that to the  the last subsection of this paper.
\end{Rema}
 \subsection{Proof of Theorem \ref{thm-mesure-2}: All initial data are finite measures}In this subsection, we shall outline the proof of Theorem \ref{thm-mesure-2}. More precisely, we will focus on the local well-posedness matter then we   give the details for the proof of the global estimates at the end of this section.\\
 
  As we pointed out before, the most challenging part is how to give a rigorous and suitable sense to the initial data of the quantity $r\rho$ if the initial density $\rho_0$ is only a finite measure. Also, it is important to make a choice that does not perturb the continuity of the solution near $t=0$\footnote{Or at least the weak continuity near $t=0$ as we will see later on.}. In the case where $\rho_0$ is an axisymmetric function in $ L^1(\mathbb{R}^3)$, we saw in the proof of Theorem \ref{Second-Th.} that $\widetilde{\rho}_0=r\rho_0 \in L^1({\Omega})$ with 
  \begin{equation*}
  \|\widetilde{\rho}_0 \|_{L^1 (\Omega)} =\frac{1}{2\pi} \|\rho_0 \|_{L^1(\mathbb{R}^3)}.
  \end{equation*}
  Hence, the general case where $\rho_0 $ is axisymmetric measure should fulfill this properties as well. More precisely, if $\rho_0$ is a finite axisymmetric measure  in $\mathscr{M}(\RR^3)$, then,  we should look for a measure $\widetilde{\rho}_0$ in $\mathscr{M}(\Omega)$ that satisfies 
\begin{equation}
\|\widetilde{\rho}_0 \|_{\mathscr{M}(\Omega)} = \frac{1}{2\pi} \|\rho \|_{\mathscr{M}(\RR^3)}.
\end{equation}   The best candidate is then inspired by Proposition \ref{prop:mu-axi}. More precisely,   we shall define  $\widetilde{\rho}_{0}$ as
\begin{equation}\label{rho_0:2:def}
\left\{ \begin{array}{c}
\left\langle \widetilde{\rho}_0, \psi\right\rangle \triangleq  \frac{1}{2\pi} \displaystyle\int_{\RR^3} \phi_\psi d\rho_0,\qquad\forall\psi\in C^0(\Omega)\\~~\\
\phi_\psi(x,y,z) \triangleq \psi(\sqrt{x^2+y^2},z),
\end{array}\right.
\end{equation} 
The factor $\frac{1}{2\pi}$ is added for a compatibility reason\footnote{see identity \eqref{inspiration:EQ} which is why we should define $\widetilde{\rho}_0$ by \eqref{rho_0:2:def}.} and all the results of Proposition \ref{prop:mu-axi} and the remark thereafter hold modulo that factor. \\ 

In the sequel, we denote by $C^\infty_{c,axi}(\RR^3)$ the space of axisymmetric functions $\varphi$ belonging to $C^\infty_{c}(\RR^3)$ and satisfying the following boundary conditions
$$\varphi|_{r=0}= \partial_r\varphi|_{r=0} =0. $$
For such test function $\varphi$,  we also adopt the  identification
 $\varphi\circ \mathfrak{F} \approx \psi $, where $ \mathfrak{F}$  is defined by \eqref{F-r-z-theta}.
Moreover, for simplicity we write
$$\left\langle f,\psi \right\rangle_{\Omega} $$
 instead of 
  $$\left\langle f,\varphi\circ \mathfrak{F}  \right\rangle_{\Omega}, $$
for any distribution $f $ on $\Omega$. Let us consider $\mu $ to be any measure in $\mathcal{M}(\Omega)$ and
we set the  goal of this part to the understanding of the following integral system 
\begin{equation}\label{int-equation-B}
\left\{\begin{array}{ll}
\omega_{\theta}(t)=\Ss_1(t)\omega_{0}-\int_{0}^{t}\Ss_1(t-\tau)\Div_{\star}\big(v(\tau)\omega_{\theta}(\tau)\big)d\tau-\int_{0}^{t}\Ss_1(t-\tau)\partial_{r}\rho(\tau) d\tau, &\\~~\\
\widetilde{\rho}(t)=\Ss_1(t)\mu-\int_{0}^{t}\Ss_1(t-\tau)\Div_{\star}\big(v(\tau)\widetilde{\rho}(\tau)\big)d\tau-2\int_{0}^{t}\Ss_1(t-\tau)\partial_{r}\rho(\tau) d\tau, &\\~~\\
\rho(t)=\Ss_2(t)\rho_0-\int_{0}^{t}\Ss_2(t-\tau)\Div\big(v(\tau)\frac{\widetilde{\rho}}{r} (\tau)\big)d\tau.
\end{array}
\right.
\end{equation}
Above,  for $t>0$ and $(r,z)=(\sqrt{x^2+ y^2},z)$,    we consider the identification of axisymmetric functions $$\omega_\theta=\omega_\theta(t,x,y,z)=\omega_\theta(t,r,z), \quad \rho=\rho(t,x,y,z)=\rho(t,r,z)$$   and $\widetilde{\rho} $ is, for now, an unknown function of the form $\widetilde{\rho}= \widetilde{\rho}(t,r,z)$. 
Remark that the system \eqref{int-equation-B} is equivalent to $\eqref{B(mu,kappa)}$ if $(\omega_\theta,\widetilde{\rho},\rho)$ is regular enough and if $\widetilde{\rho}=r\rho$ and $ \mu=r\rho_0$, at least for integrable initial density. \\

 The following theorem, which is the main result of this section, is a general version of the local results in Theorem \ref{thm-mesure-2}.
\begin{Theo}\label{Th2:general version}
Let $(\omega_0,\rho_0,\mu)$ be in $\mathscr{M}(\Omega) \times \mathscr{M}(\RR^3) \times \mathscr{M}(\Omega) $, such that $\rho_0$ is axisymmetric in the sense of Definition \ref{def:axi:measure}. Then, the following hold

\begin{itemize}
\item[(i)]\textbf{Local well-posedness of \eqref{int-equation-B}.} There exists a non negative constant $\varepsilon$ such that, if 
\begin{equation}\label{petitness:1}
 \norm{\omega_{0,pp} }_{\mathscr{M}(\Omega)}+\norm{\mu_{pp} }_{\mathscr{M}(\Omega)}+\norm{\rho_{0,pp} }_{\mathscr{M}(\RR^3)}  \leq \EE, 
\end{equation}
then, there exists $T=T(\omega_0,\rho_0,\mu)>0$ for which \eqref{int-equation-B} has a unique solution, defined on $ [0,T]$, and satisfying, for all $p\in [1,\infty]$
\begin{equation}\label{Lp-es:Th2:V2}
\sup_{t\in (0,T]}  \; \left\{ t^{1-\frac{1}{p}} \norm{\left(\omega(t),\widetilde{\rho}(t) \right)}_{L^p(\Omega)\times L^p(\Omega) } +  t^{\frac{3}{2}( 1-\frac{1}{p})} \norm{\rho(t)}_{L^p(\RR^3)}\right\}\lesssim  \norm{\left(\omega_{0 } ,\mu \right)}_{\mathscr{M}(\Omega)\times\mathscr{M}(\Omega) } +\norm{\rho_{0 } }_{\mathscr{M}(\RR^3)}   .
\end{equation}  
\item[(ii)] \textbf{Weak convergence to the initial data.} For all $\varphi\in C^\infty_{c,Axi}(\RR^3)$,   we have\footnote{We recall that we are using the identification $\psi \approx \varphi \circ \mathfrak{F}.$}
\begin{equation}
\lim_{t\rightarrow 0}\left\langle \omega_\theta(t) |  \psi  \right\rangle_{\Omega} =\left\langle \omega_0 |  \psi  \right\rangle_{\Omega}  
\end{equation}
\begin{equation}
\lim_{t\rightarrow 0}\left\langle \widetilde{\rho}(t) |  \psi  \right\rangle_{\Omega} =\left\langle\mu |  \psi  \right\rangle_{\Omega}  ,
\end{equation}
\begin{equation}
\lim_{t\rightarrow 0}\left\langle  \rho (t) | \varphi \right\rangle_{\RR^3} =\left\langle  \rho _0 | \varphi \right\rangle_{\RR^3} . 
\end{equation}
\item[(iii)]\textbf{Local well-posedness of the Boussinesq system \eqref{B(mu,kappa)}.} Moreover, if $\mu=\widetilde{\rho}_0$ is given by \eqref{rho_0:2:def}, then the condition on the size of the initial data \eqref{petitness:1} can be replaced by 
\begin{equation}\label{petitness:2}
 \norm{\omega_{0,pp} }_{\mathscr{M}(\Omega)} +\norm{\rho_{0,pp} }_{\mathscr{M}(\RR^3)}  \leq \widetilde{\EE}, 
\end{equation}
for some $\widetilde{\EE}>0$. Also,  we have 
$$\lim_{t\rightarrow 0}\left\langle \widetilde{\rho}(t) |  \psi  \right\rangle_{\Omega} = \lim_{t\rightarrow 0}\left\langle  r\rho (t) | \psi \right\rangle_{\Omega} = \left\langle \widetilde{\rho}_0 |  \psi  \right\rangle_{\Omega}  , \quad \forall \psi \in C^\infty_c(\Omega), $$
$$\widetilde{\rho}(t) = r\rho(t), \quad \forall t>0 $$
and $(\omega_\theta,\rho)$ is actually the unique solution of the Boussinesq system \eqref{B(mu,kappa)} on $[0,T]$.
\end{itemize}
\end{Theo} 
\begin{proof}We  prove the results of the above theorem in the order given in its statement\\
$\bullet$ \textbf{Proof of (i): Local well-posedness of system \eqref{int-equation-B}.} \\
We have to prove the existence of some $T>0$, and a unique solution $(\omega_\theta,\widetilde{\rho},\rho)\in X_T\times X_T \times Z_T$ to \eqref{int-equation-B}. This can be done by a fixed point argument, more precisely, by following exactly the same idea explored in the proof of Theorem \ref{Second-Th.}. To do so, the free part    $(\Ss_1(t)\omega_0,\Ss_1(t) \mu, \Ss_2(t)\rho_0)$ has to be small enough in $X_T\times X_T \times Z_T$, as $T$ is close to zero and the nonlinear parts have to be estimated by using the properties of the semi-groups stated in the  subsection \ref{Semi-group estimates}. Indeed, by employing the results of the subsection \ref{Semi-group estimates} we can get the same estimates obtained in the proof of \mbox{Proposition \ref{existence Locale-mu}}
\begin{equation}\label{estimates in X_T*X_T*Z_T}
\left\{ \begin{array}{l} 
\| \omega_{\theta} \|_{X_T} \leq \| \Ss_1(\cdot) \omega_0\|_{X_T} + C\| \omega_{\theta} \|_{X_T} ^2 + C \|{\rho} \|_{Z_T}\\~~\\
\| \widetilde{\rho} \|_{X_T} \leq \|\Ss_1(\cdot)  \mu \|_{X_T} + C\| \omega_{\theta} \|_{X_T}\| \widetilde{\rho} \|_{X_T}  + C \|{\rho} \| _{Z_T}\\~~\\
\|  \rho  \|_{Z_T} \leq \|\Ss_2(.) \rho_0 \|_{Z_T}+ C\| \omega_{\theta} \|_{X_T}\| \widetilde{\rho} \|_{X_T},
   \end{array} \right.
\end{equation}
for some universal constant $C>0$. The system \eqref{estimates in X_T*X_T*Z_T} yields then to the following estimate, up to a suitable modification in $C$ 
\begin{equation}\label{poly}
A_T \triangleq \| \omega_{\theta} \|_{X_T}+ \| \widetilde{\rho} \|_{X_T} \leq  A_{0,T} + C A_T^2,
\end{equation}
where, $A_{0,T} $ is given by
\begin{equation}
A_{0,T} \bydef \|\Ss_1(\cdot) \omega_0\|_{X_T}  + \|\Ss_1(\cdot) \mu\|_{X_T} + C \|\Ss_2(.) \rho_0 \|_{Z_T}.
\end{equation} 
The local well-posedness follows then by usual arguments if $\displaystyle\lim_{T\rightarrow 0} A_{0,T}$ is small enough.

Now, in order to measure the size of $ A_{0,T}$ for small $T$, we use Proposition  \ref{prop mu-lamda} and Proposition \ref{3D-heat kernel} to get 
\begin{equation}
\lim_{T\rightarrow 0 }A_{0,T} \leq \norm{\omega_{0,pp} }_{\mathscr{M}(\Omega)}+ \norm{\mu_ {pp} }_{\mathscr{M}(\Omega) }  +C \norm{\rho_{0,pp} }_{\mathscr{M}(\RR^3)},
\end{equation}
which gives, for some $\widetilde{C}>0$
\begin{equation}\label{smalness:kappa}
\lim_{T\rightarrow 0 }A_{0,T} \leq \widetilde{C} \big( \norm{\omega_{0,pp} }_{\mathscr{M}(\Omega)} + \norm{\mu_ { pp} }_{\mathscr{M}(\Omega) }+\norm{\rho_{0,pp} }_{\mathscr{M}(\RR^3)} \big).
\end{equation}  
Thus, if the r.h.s. of the last inequality above is small enough, then the fixed point argument guarantees the local well-posedness of \eqref{int-equation-B}. That is to say, there exist $\EE>0$, such that if 
\begin{equation}\label{smalness:epsilon}
 \norm{\omega_{0,pp} }_{\mathscr{M}(\Omega)}+ \norm{\mu_ {pp} }_{\mathscr{M}(\Omega) }  +\norm{\rho_{0,pp} }_{\mathscr{M}(\RR^3)}  \leq \EE,
\end{equation} 
then there exists $T>0$ for which \eqref{int-equation-B} has a unique solution $(\omega_\theta,\widetilde{\rho},\rho)$ in $X_T\times X_T \times Z_T. $\\
Remark that the fixed point argument gives in particular the estimate \eqref{Lp-es:Th2:V2} for $p=\frac{4}{3}$. The proof of estimate \eqref{Lp-es:Th2:V2} for all $p\in [1,\infty]$ can be done by a Bootstrap argument. The details of that are exactly the same as in the proof of Proposition \ref{properties-decay-Lp}. Assertion (i) is then proved.\\
$\bullet$ \textbf{Proof of (ii): Weak convergence to the initial data.}\\ 
Let us introduce the following linear operators
$$ \mathcal{F}_1(f)(t) = \int_{0}^t \Ss_1(t-\tau) \Div_{\star}(v f)(\tau) d\tau, $$
$$ \mathcal{F}_2(g) (t)= \int_{0}^t \Ss_2(t-\tau) \Div(v g)(\tau) d\tau  $$
and
$$ \mathcal{G}(\rho)(t)= \int_{0}^t\Ss_1(t-\tau) \partial_r \rho(\tau) d\tau,$$
where, $v$ is the velocity associated with the unique solution $(\omega_\theta,\widetilde{\rho},\rho)$ constructed in the previous step.
Hence, our integral system \eqref{int-equation-B} can be rewritten as
 \begin{equation}\label{Duhamel-2}
\left\{ \begin{array}{l} 
 \omega_\theta(t) = \Ss_1(t) \omega_0 - \mathcal{F}_1(\omega_\theta)(t) - \mathcal{G}(\rho)(t)\\~~\\
\widetilde{\rho}(t) = \Ss_1(t) \mu- \mathcal{F}_1(\widetilde{\rho})(t) - 2\mathcal{G}(\rho)(t)\\~~\\
 \rho(t)= \Ss_2(t) \rho_0 - \mathcal{F}_2(\frac{\widetilde{\rho}}{r})(t)
 ,  \end{array} \right.
\end{equation}
First, we point out that, for every $\varphi\in C^\infty_c(\RR^3)$ and $\psi\in C^\infty_c(\Omega)$, we have
\begin{equation*}
\lim_{t\rightarrow 0}\int_{\Omega}\Ss_1(t)\omega_0 \psi drdz = \int_{\Omega}\psi(r,z)  d(\omega_0(r,z)),
\end{equation*}
\begin{equation*}
\lim_{t\rightarrow 0}\int_{\Omega}\Ss_1(t)\mu\psi drdz = \int_{\Omega}  \psi(r,z)   d(\mu(r,z)),
\end{equation*}
\begin{equation*}
\lim_{t\rightarrow 0}\int_{\RR^3}\Ss_2(t)\rho _0 \varphi dx = \int_{\RR^3}\varphi(x) d( \rho_0 (x)).
\end{equation*} 
Thus, in order to prove the convergence to the initial data in \eqref{Duhamel-2}, we need to show that the terms containing the operators $\mathcal{F}_{1}$, $\mathcal{F}_{2}$ and $\mathcal{G}$ tend weakly to zero as $t$ goes to $0$. Let us begin by proving that
\begin{equation}\label{claim-F_2 limit}
\lim_{t\rightarrow 0}\int_{\RR^3}\mathcal{F}_2(\frac{\widetilde{\rho}}{r})(t) \varphi(x) dx = 0.
\end{equation}

 Remark that the operators $\Div$ and $\Ss_2(t)$ commute, whereas an integration by parts followed by the Proposition \ref{Prop-W-E-1} and  Biot--Savart law yield, in view of the notation $\widetilde{\nabla \varphi}  = (\nabla \varphi)\circ\mathfrak{F}$
\begin{align*}
\bigg|\int_{\RR^3}\mathcal{F}_2(\frac{\widetilde{\rho}}{r}) \varphi(x) dx   \bigg|& \lesssim \int_{0}^t \int_\Omega\big| r\Ss_2(t-\tau) (v\frac{\widetilde{\rho}}{r}) \cdot \widetilde{\nabla \varphi}(r,z)  \big|drdz d\tau\\
&\lesssim \int_{0}^t   \| \omega_\theta(\tau) \|_{L^\frac{4}{3}(\Omega)} \|\widetilde{\rho}(\tau)\|_{L^\frac{4}{3}(\Omega)}d\tau \| \widetilde{\nabla \varphi}\|_{L^\infty(\Omega)}\\
&\lesssim \int_0^t\frac{d\tau}{\tau^\frac{1}{2}} \|\omega_\theta\|_{X_T}\| \widetilde{\rho}\|_{X_T}\|\widetilde{\nabla \varphi}\|_{L^\infty(\Omega)}\\
&\lesssim t^\frac{1}{2} \|\omega_\theta\|_{X_T}\| \widetilde{\rho}\|_{X_T}\|\widetilde{\nabla \varphi}\|_{L^\infty(\Omega)},
\end{align*} 
This is enough to guarantee \eqref{claim-F_2 limit}.\\
For the rest of the limits, we will restrict our selves to the ones appearing in the equation of $\widetilde{\rho}$ due to the similarity of the equations of $\widetilde{\rho}$ and $\omega_\theta.$ Let us point out first that the operators $\partial_r$ and $\Ss_1(t)$ do not commute. To overcome this issue, let us rewrite the equation of $\widetilde{\rho}$ in terms of $\Ss_2$.   To do so, owing to the fact that
 $$ \Div_{\star}(v\widetilde{\rho})  = \Div(v\widetilde{\rho}) - \frac{v^r}{r} \widetilde{\rho},$$
 then the equation of   $\widetilde{\rho}$, given by
\begin{equation*}
\partial_t\widetilde{\rho} - \Delta \widetilde{\rho} + \Div_{\star}(v\widetilde{\rho}) + \frac{\widetilde{\rho}}{r^2}= -2\partial_r\rho,
\end{equation*}
can be written in the integral form   as
\begin{align}\label{Duhamel-2-rrho}
\widetilde{\rho}(t)= \Ss_2(t)\mu - \int_0^t \Ss_2(t-\tau)\Div( v\widetilde{\rho}) d\tau &+ \int_0^t \Ss_2(t-\tau) \big( \frac{v^r}{r} \widetilde{\rho}\big) d\tau - \int_0^t \Ss_2(t-\tau)   \frac{\widetilde{\rho}}{r^2}d\tau \nonumber\\&-2 \int_0^t \Ss_2(t-\tau)\partial_r \rho  d\tau .
\end{align} 
Except of the first term on the r.h.s above, all the rest of terms should go to $0$ (in distributional sense) in order to reach our claim. 
Indeed, by using the fact that the operator $\Div$ commutes with $\Ss_2(t)$, Proposition \ref{Prop-W-E-1} and the Biot-Savart law, we obtain
\begin{align*}
\bigg| \int_0^t \int_\Omega\Ss_2(t-\tau)\Div( v\widetilde{\rho}) \psi(r,z) drdzd\tau   \bigg|&\lesssim  \int_{0}^t \int_\Omega\big|\Ss_2(t-\tau) (v \widetilde{\rho} ) \cdot \nabla \psi(r,z)  \big|drdz d\tau\\
&\lesssim \int_{0}^t   \| \omega_\theta(\tau) \|_{L^\frac{4}{3}(\Omega)} \|\widetilde{\rho}(\tau)\|_{L^\frac{4}{3}(\Omega)}d\tau \|\nabla \psi\|_{L^\infty(\Omega)}.
\end{align*}
We continue as in the proof of \eqref{claim-F_2 limit} to obtain 
$$\big| \int_0^t \int_\Omega\Ss_2(t-\tau)\Div( v\widetilde{\rho}) \psi(r,z) drdzd\tau   \big| \lesssim  t^\frac{1}{2} \|\omega_\theta\|_{X_T}\| \widetilde{\rho}\|_{X_T}\|\nabla\psi\|_{L^\infty(\Omega)},$$
which tends to $0$ as $t$ goes to $0$. For the $3^{rd}$ term on the r.h.s of \eqref{Duhamel-2-rrho}, we proceed by using again Proposition \ref{Prop-W-E-1}, and the Biot--Savart law to infer that
\begin{align*}
\bigg| \int_0^t \int_\Omega\Ss_2(t-\tau) \big( \frac{v^r}{r} \widetilde{\rho}\big) \psi(r,z) drdzd\tau   \bigg|& = \bigg| \int_0^t \int_\Omega r \Ss_2(t-\tau) \big( \frac{v^r}{r} \widetilde{\rho}\big) \frac{\psi(r,z)}{r} drdzd\tau   \bigg|\\
&\lesssim \int_{0}^t   \| \omega_\theta(\tau) \|_{L^\frac{4}{3}(\Omega)} \|\widetilde{\rho}(\tau)\|_{L^\frac{4}{3}(\Omega)}d\tau \|\frac{\psi }{r}\|_{L^\infty(\Omega)}.
\end{align*}
Again, we continue as in the proof of \eqref{claim-F_2 limit} to get 
\begin{equation*}
\bigg| \int_0^t \int_\Omega\Ss_2(t-\tau)\big( \frac{v^r}{r} \widetilde{\rho}\big)  \psi(r,z) drdzd\tau   \bigg| \lesssim  t^\frac{1}{2} \|\omega_\theta\|_{X_T}\| \widetilde{\rho}\|_{X_T}\|\frac{\psi}{r}\|_{L^\infty(\Omega)},
\end{equation*}  
which tends to $0$ as $t$ goes to $0$.
For the $4^{th}$ term on the r.h.s of \eqref{Duhamel-2-rrho}, we use again Proposition \ref{Prop-W-E-1} and the Biot--Savart law to find that
\begin{align*}
\bigg| \int_0^t \int_\Omega\Ss_2(t-\tau)\frac{\widetilde{\rho}}{r^2}\psi(r,z) drdzd\tau   \bigg|& = \bigg| \int_0^t \int_\Omega r^2 \Ss_2(t-\tau)\frac{\widetilde{\rho}}{r^2}  \frac{\psi(r,z)}{r^2} drdzd\tau   \bigg|\\
&\lesssim \int_{0}^t\|\widetilde{\rho}(\tau)\|_{L^\frac{4}{3}(\Omega)}d\tau \|\frac{\psi }{r^2}\|_{L^4(\Omega)}\\
&\lesssim \int_{0}^t   \frac{d\tau}{\tau^\frac{1}{4}} \|\widetilde{\rho} \|_{X_T} \|\frac{\psi }{r^2}\|_{L^4(\Omega)}\\
&\lesssim t^\frac{3}{4} \|\widetilde{\rho} \|_{X_T} \|\frac{\psi }{r^2}\|_{L^4(\Omega)},
\end{align*}
which tends to $0$ as $t$ goes to $0$. Finally, for the last term in \eqref{Duhamel-2-rrho}, implementing again Proposition \ref{Prop-W-E-1} and Biot--Savart law yield
\begin{align*}
\bigg| \int_0^t \int_\Omega\Ss_2(t-\tau)\partial_r \rho\psi(r,z) drdzd\tau\bigg|& = \bigg| \int_0^t \int_\Omega r ^\frac{3}{4} \Ss_2(t-\tau)\partial_r \rho \frac{\psi(r,z)}{r^\frac{3}{4}} drdzd\tau   \bigg|\\
&\lesssim \int_{0}^t   \frac{1}{\tau^\frac{1}{2}}  \|r^\frac{3}{4}\rho (\tau)\|_{L^\frac{4}{3}(\Omega)}d\tau \|\frac{\psi }{r^\frac{3}{4} }\|_{L^4(\Omega)}\\
&\lesssim \int_{0}^t   \frac{1}{\tau^\frac{1}{2}}  \| \rho (\tau)\|_{L^\frac{4}{3}(\RR^3)}d\tau \|\frac{\psi }{r^\frac{3}{4} }\|_{L^4(\Omega)}\\
&\lesssim \int_{0}^t   \frac{d\tau}{\tau^\frac{7}{8}} \| \rho \|_{Z_T} \|\frac{\psi }{r^\frac{3}{4}}\|_{L^4(\Omega)}\\
&\lesssim t^\frac{1}{8} \| \rho  \|_{Z_T} \|\frac{\psi }{r^\frac{3}{4} }\|_{L^4(\Omega)},
\end{align*}
which tends to $0$ as $t$ goes to $0$. All in all, we deduce  that $\widetilde{\rho}(t)$ tends to $\mu$ (in distributional sense) as $t$ goes to $0$. Similar arguments can be used to prove that $ \omega_\theta(t)$ tends to $\omega_0$ when $t$ goes to $0$. The details are left to the reader. 
\begin{Rema}\label{remark:test:function}
We should point out that the computations of this step hold true whenever the test function $\psi$ is in $C^1(\Omega)$ such that
 \begin{equation}\label{test:func:condition}
\|\nabla \psi \|_{L^\infty(\Omega)}+ \|\frac{\psi}{r}\|_{L^\infty(\Omega)}+\|\frac{\psi}{r}\|_{L^4(\Omega)}+\|\frac{\psi}{r^2}\|_{L^4(\Omega)} <\infty.
 \end{equation}
Such a condition is automatically satisfied if we take $\psi = \varphi \circ \mathfrak{F}$, for any $\varphi \in C^\infty_{c, Axi}(\RR^3)$. Indeed, the boundary conditions on the test function belonging to $C^\infty_{c, Axi}(\RR^3)$, together with Taylor expansion near $r=0$ would clearly imply \eqref{test:func:condition}.
\end{Rema}

$\bullet$ \textbf{Proof of (iii): Local well-posedness of the Boussinesq system ($\mu=\widetilde{\rho}_0$).}\\
Now, we assume that $\mu $ and $\rho_0$ are connected by the formula \eqref{rho_0:2:def}. That is, we set $\mu=\widetilde{\rho}_0$, where 
\begin{equation}\label{rho_0:2:def:2222}
\left\{ \begin{array}{c}
\left\langle \widetilde{\rho}_0, \psi\right\rangle \triangleq  \frac{1}{2\pi} \displaystyle\int_{\RR^3} \phi_\psi d\rho_0,\quad \forall \psi\in C^0(\Omega),\\~~\\
\phi_\psi(x,y,z) \triangleq \psi(\sqrt{x^2+y^2},z).
\end{array}\right.
\end{equation}   
Proposition \ref{prop:mu-axi} and the remark thereafter yield, for all $\varphi\in C^\infty_{c,Axi}(\RR^3)$
\begin{equation}\label{EQ1:000}
\frac{1}{2\pi} \left\langle \rho_0 | \varphi\right\rangle_{\RR^3} = \left\langle \widetilde{\rho}_0 | \varphi\circ \mathfrak{F} \right\rangle_\Omega=\left\langle \widetilde{\rho}_0 | \psi \right\rangle_\Omega,
\end{equation} 
and 
$$\norm {\widetilde{\rho}_{0,pp}}_{\mathscr{M}(\Omega)} =\frac{1}{2\pi} \norm { \rho_{0,pp} }_{\mathscr{M}(\RR^3)}. $$
This estimate on the size of $\norm {\widetilde{\rho}_{0,pp}}_{\mathscr{M}(\Omega)}$ can be used then in \eqref{smalness:kappa} to obtain, for some $ \widetilde{C}_0>0$
$$ \lim_{T\rightarrow 0 }A_{0,T} \leq  \widetilde{C}_0 \big( \norm{\omega_{0,pp} }_{\mathscr{M}(\Omega)} +\norm{\rho_{0,pp} }_{\mathscr{M}(\RR^3)} \big). $$
 Hence, it is obvious that, up to a modification in $\varepsilon$ given by \eqref{smalness:epsilon}, then \eqref{smalness:epsilon} can be replaced by 
  \begin{equation}\label{EQ3:000}
 \norm{\omega_{0,pp} }_{\mathscr{M}(\Omega)} + \norm{\rho_{0,pp} }_{\mathscr{M}(\RR^3)} \leq \widetilde{\varepsilon}. 
  \end{equation} 
 The local well-posedness of \eqref{int-equation-B} is then guaranteed as long as \eqref{EQ3:000} is satisfied. This ends the proof of the first part of (iii).  Now, since $\rho(t)$ is axisymmetric, belonging to $L^1(\RR^3)$ for all $t>0$, then $r\rho(t)$ belongs to $L^1(\Omega)$ and a change of variables gives 
\begin{equation}\label{EQ2:000}
\frac{1}{2\pi} \left\langle \rho(t) | \varphi \right\rangle_{\RR^3} =\left\langle r\rho(t)| \varphi\circ \mathfrak{F}\right\rangle_{\Omega}. 
\end{equation}

 Hence, the weak limits (as $t$ tends to $0$) proved in the previous step, together with \eqref{EQ1:000} and \eqref{EQ2:000}  yield
 \begin{equation}\label{inspiration:EQ}
  \lim_{t\rightarrow 0} \left\langle r\rho(t)| \varphi\circ \mathfrak{F}\right\rangle_{\Omega}  = \frac{1}{2\pi}\lim_{t\rightarrow 0} \left\langle \rho(t) | \varphi \right\rangle_{\RR^3} = \frac{1}{2\pi}\left\langle \rho_0 | \varphi \right\rangle_{\RR^3} = \left\langle \widetilde{\rho}_0 | \varphi\circ \mathfrak{F} \right\rangle_\Omega.
\end{equation}  
 
Consequently, for all $\psi\in C^\infty_{c }(\Omega)$, we obtain 
\begin{equation}\label{weak:lim:r rho}
 \lim_{t\rightarrow 0} \left\langle r\rho(t)| \psi\right\rangle_{\Omega} =   \left\langle \widetilde{\rho}_0 | \psi\right\rangle_\Omega= \lim_{t \rightarrow 0}\left\langle \widetilde{\rho}(t) | \psi \right\rangle_\Omega. 
\end{equation} 
Moreover, remark that for all $t>0$, the quantity $\sigma(t) \triangleq r\rho(t)$ satisfies the equation 
$$ \partial_t \sigma - \left( \Delta - \frac{1}{r^2}\right)\sigma +  \Div_{\star}(v \widetilde{\rho} ) = - 2 \partial_r \rho.$$
This, together with \eqref{weak:lim:r rho} yield to the following system for $\sigma $
\begin{equation}
\left\{\begin{array}{ll}
\partial_t \sigma - \left( \Delta - \frac{1}{r^2}\right)\sigma +  \Div_{\star}(v \widetilde{\rho} ) = - 2 \partial_r \rho, & (t,r,z)\in \RR^+_*\times \Omega,\\
\sigma_{|_{t=0}} = \widetilde{\rho}_0,
\end{array}\right.
\end{equation}
where the initial condition is to be understood in the weak sense given by \eqref{weak:lim:r rho}.
We recall, on the other hand, that $\widetilde{\rho}$ satisfies the system 
\begin{equation}
\left\{\begin{array}{ll}
\partial_t \widetilde{\rho} - \left( \Delta - \frac{1}{r^2}\right)\widetilde{\rho} +  \Div_{\star}(v \widetilde{\rho} ) = - 2 \partial_r \rho, & (t,r,z)\in \RR^+_*\times \Omega,\\
\widetilde{\rho}_{|_{t=0}} = \widetilde{\rho}_0
\end{array}\right.
\end{equation}
One finds then that the quantity $ \widetilde{\rho}- r\rho $ satisfies a heat equation with zero inputs. It is easy then to deduce that $r\rho(t)= \widetilde{\rho}(t),$ for all $t>0.$\\
We emphasize that this characterization of $\widetilde{\rho}$ would imply that $(\omega_\theta,\rho)$ solves the Boussinesq system. Moreover, all the estimates proved for $\widetilde{\rho}$ hold  for for the quantity $ r\rho.$ 
Theorem \ref{Th2:general version} is then proved.
\end{proof}
\subsection{Global well-posedness }
The results proved in Theorem \ref{Th2:general version} provide information only on the local well-posedness  whenever the initial data $(\omega_0,\rho_0)$ is suitable and lies in $\mathscr{M}(\Omega)\times \mathscr{M} (\RR^3)$. However, one can in fact extend the local solution to be defined for all $t>0$.   Indeed, from the proof of Theorem \ref{Th2:general version}, we   deduce that there exists $t_0\in (0,T)$ such $ (\omega(t_0), \rho(t_0))\in  L^1(\Omega) \times L^1(\RR^3)$. Hence, Theorem \ref{First-Th.} insures the existence of a unique solution of the Boussinesq system with initial data $(\omega(t_0), \rho(t_0))$, denoted for now by $(\bar{\omega},\bar{\rho})$. This solution is defined on $[t_0,\infty)$ and satisfies in particular, for all $p\in [1,\infty]$ 
$$\sup_{t\geqslant t_0} (t-t_0)^{1-\frac{1}{p}} \|(\bar{\omega}(t), r\bar{\rho}(t) )\|_{L^p(\Omega)\times L^p(\Omega)}+\sup_{t\geqslant t_0} (t-t_0)^{\frac{3}{2}(1-\frac{1}{p})} \| \bar{\rho}(t) \|_{L^p(\RR^3)} < \infty. $$
On the other hand, the local solution $(\omega,\rho)$ constructed in Theorem \ref{Th2:general version} satisfies, for all $p\in [1,\infty]$
$$\sup_{t\in [t_0,T]} (t-t_0)^{1-\frac{1}{p}} \|( \omega (t), r \rho (t) )\|_{L^p(\Omega)\times L^p(\Omega)}+\sup_{t\in [t_0,T]} (t-t_0)^{\frac{3}{2}(1-\frac{1}{p})} \|  \rho (t) \|_{L^p(\RR^3)} < \infty. $$
To conclude, we only need to use the above estimates   and repeat the arguments leading to the uniqueness in the proof of  Theorem \ref{Th2:general version} to end up with $(\omega,\rho) \equiv (\bar{\omega},\bar{\rho})$ on $[t_0,T].$ This proves that the local solution is uniquely extendable to a global one.


\end{document}